\documentclass[reqno]{amsart}

\usepackage{amssymb}
\usepackage{fancybox}
\usepackage{xcolor}
\usepackage{appendix}
\usepackage{pifont}
\usepackage{mathrsfs}
\usepackage{leftidx}   
\usepackage{comment}   

\usepackage{natbib}

\usepackage{enumitem} 
\setlist[itemize]{noitemsep} 
\usepackage{enumitem}  

\usepackage{times}  

\newtheorem{theorem}{Theorem}[section]
\newtheorem{lemma}[theorem]{Lemma}

\newtheorem{proposition}[theorem]{Proposition}
\newtheorem{assumption}[theorem]{Assumption}
\newtheorem{remark}[theorem]{Remark}

\numberwithin{equation}{section}

\newcommand{\abs}[1]{\lvert#1\rvert}

 \newcommand{\rom}[1]{
  \textup{\lowercase\expandafter{\romannumeral#1}} }
  
\newcommand{\Rom}[1]{
  \textup{\uppercase\expandafter{\romannumeral#1}} }
  
\mathchardef\mhyphen="2D

\newcommand{\reals}{{\mathbb R}}
\newcommand{\bbr}{\reals}
\newcommand{\bbz}{{\mathbb Z}}
\newcommand{\bbn}{{\mathbb N}}
\newcommand{\PP}{\mathbb{P}}
\newcommand{\EE}{\mathbb{E}}

\newcommand{\one}{{\bf 1}}
\newcommand{\vep}{\varepsilon}

\newcommand{\eid}{\stackrel{d}{=}}

\newcommand{\calF}{\mathcal{F}}

\begin{document}

\title[Extremes of semi-exponential processes]{A new shape of extremal
  clusters for certain stationary semi-exponential processes with
  moderate long range dependence}

\author{Zaoli Chen}

\address{Department of Mathematics, Cornell University, Ithaca, New York 14850, USA}

\email{zc288cornell.edu}

\author{Gennady Samorodnitsky}
\address{School of Operation Research and Information Engineering,
Cornell University, Ithaca, New York 14850, USA}
\email{gs18@cornell.edu}
\thanks{This research was partially supported by the ARO grant 
 W911NF-18 -10318 and the NSF grant
  DMS-1506783 at Cornell  University}

\subjclass[2000]{Primary 60G70, 60F17; Secondary 60G57.}


\keywords{Extreme value theory, long range dependence, random
  sup-measure, stable regenerative set, subexponential distributions,
  semi-exponential distributions, Gumbel maximum domain of attraction.}

\begin{abstract}
Extremal clusters of stationary processes with long memory can be
quite intricate. For certain stationary infinitely
divisible processes with subexponential tails, including both
power-like tails and certain lighter tails, e.g. lognormal-like
tails. such  clusters may take the shape of  stable regenerative
sets. In this paper  we show that for semi-exponential tails, which
are even lighter, a new shape of extremal clusters arises. In this
case each stable regenerative set supports a random panoply of varying
extremes.  
\end{abstract}

\maketitle


\section{Introduction} \label{sec;Intro}

In this paper, we study the extremes of certain stationary infinitely
divisible processes with  long range dependence. The marginal
log-tails are, roughly, of the order $-x^{-\alpha}$ for some
$0<\alpha<1$. We call such tails semi-exponential, and the exact
definition and the assumptions are given in
Section \ref{sec;SemiExpPro}. Such marginal tails are in the Gumbel
maximum domain of attraction, and the assumptions we impose will also
guarantee that these tails are subexponential.

Extremal limit theorems for such processes are interesting from
several points of view. First of all, how do the extreme values of
such processes cluster? Second, to what extent does the ``single large
jump'' heuristic hold for such processes? This principle usually governs both extreme values and large deviations of
weakly dependent subexponential stochastic systems.  
Extremal clusters appear when the values of the process at distinct
time points are tail (asymptotically) dependent; see
\cite{resnick:2007}. The size  of extremal clusters is a
subject of long-standing interest, and the notion of extremal index due to
\cite{davis:1982} and \cite{leadbetter:1983} is specifically designed
to quantify this size. When the stationary process has long memory
affecting the extreme values of the process in the sense of
\cite{samorodnitsky:2016:SPLRD}, the extremal clusters may become so large
that scaling is necessary to obtain a finite limit, and then it
becomes possible to talk about the limiting shape of an extremal
cluster. For certain classes of stationary infinitely divisible
processes with subexponential tails, this limiting shape is a   random
fractal, specifically a stable
regenerative set,  supporting one extreme value.
This has been shown by \cite{lacaux:samorodnitsky:2016} in the case
when the process has regularly varying tails (which are, of course, in
the Fr\'echet maximum domain of attraction)  and by
\cite{chen:samorodnitsky:2020} in the case when the process has
certain marginal tails in 
the Gumbel maximum domain of attraction. The results of the latter
paper required, however, that these marginal tails were not too
light. In particular, \cite{chen:samorodnitsky:2020} allowed lognormal-like marginal
tails but excluded semi-exponential marginal tails, whose
limiting shape of the extremal clusters remained unknown.
In this paper, we solve this problem and characterize the limiting shape of
the extremal clusters. While the
clusters are still supported by stable regenerative sets, a new random
structure appears, in which related but different extreme values are placed in randomly
chosen locations over each stable regenerative set. 

The new shape of the extremal clusters is related to a failure of  the
``single large jump'' principle: the extreme values of the process are
caused by multiple large values of the underlying Poisson random
measure. Recall that a distribution $H$ on $[0,\infty)$ 
is called subexponential if in the usual notation for distributional
tails, 
\begin{equation} \label{eq;intro0.3}
  \lim_{x\to \infty} \frac{\overline{H\ast H}(x)}{\overline{H}(x)} =
  2. 
\end{equation}
``Single large jump'' is a widely used heuristic for processes with
subexponential tails; see
$e.g.$ \cite{foss:konstantopoulos:zachary:2007}. This
principle already fails in the case of heavier tails in the Gumbel
maximum domain of attraction, see
\cite{chen:samorodnitsky:2020}. But in the case of
semi-exponential tails, this failure is even more dramatic. In one of
our limit theorems 
we obtain a new limiting process of two parameters. It can be viewed as
a bridge between  the standard Gumbel extremal
process  of \cite{resnick:rubinovitch:1973} and the time-changed extremal
process  of \cite{chen:samorodnitsky:2020}. When the parameters of the new process tend to some of their boundary values, either of the latter two processes can be recovered.

The paper is organized as follows. Section \ref{sec;Pre} reviews the
main notions and tools we will use throughout the
paper: random closed sets, null recurrent Markov chains, distributions
in the Gumbel domain of attraction and random sup-measures.
Section \ref{sec;SemiExpPro} describes the stationary infinitely
divisible process whose extremes are to be analyzed. We state the two main
extremal limit theorems, establishing weak convergence in the spaces of
sup-measures and c\`adl\`ag functions respectively. This
section also contains most of the proofs. Several auxiliary proofs are
postponed to the  two  Appendices.

The adjective ``moderate'' decorating the term ``long range
dependence'' in the title of the paper is due to the restricted range
$\beta \in (0,1/2)$ of the parameter responsible for  the
long memory. What happens if memory becomes even longer, $i.e.$ $\beta
\in (1/2,1)$, remains a subject of future investigations, and we expect to get 
limit theorems with non-Gumbel limits. When  
 the marginal tails are regularly varying, 
non-Fr\'echet limits in this range of $\beta$ are established in
\cite{samorodnitsky:wang:2019}.

The following notation will be used  throughout the paper. We denote the set of natural numbers $\{1,2,\ldots\}$ by $\bbn$. For a nondecreasing function $H$ on $\mathbb{R}$, the inverse of
  $H$ is defined by  $ H^\leftarrow (x) = \inf \{ s : H(s) \geqslant x
  \}$, with the usual convention $\inf \emptyset = \infty$. Further, the tail of a measure $\nu$  on $\reals$ is
  $\overline{\nu}(x)=\nu(x,\infty)$. In particular 
  $\overline{F}(x)$ is the tail of a distribution $F$.
  We will use the following symbols when comparing positive sequences. 
 \begin{enumerate} [label=(\alph*)]
    \item  $a_n \sim b_n$ if 
      $\lim_{n \to \infty}a_n / b_n =1$,

     \item  $a_n \lesssim b_n $ if there exist $C>0$ such that $a_n
       \leqslant C b_n$  for large enough $n$, and analogously with
       $a_n \gtrsim b_n $, 

    \item $a_n \asymp b_n$ if both $a_n \lesssim b_n $ and $a_n
      \gtrsim b_n $.
    \end{enumerate}
    If $\{A_n\}_{n \in \mathbb{N}}$ and $\{B_n\}_{n \in \mathbb{N}}$
    are two sequences of positive random variables, we write
    \begin{enumerate} [label=(\alph*)]
    \item $A_n = o_P(B_n)$ if $A_n / B_n \to 0$ in probability, 
    \item $A_n\lesssim_P B_n$ if  $(A_n/B_n)$ is tight, and
      analogously with $A_n \gtrsim_P B_n$. 
\end{enumerate}

\section{Preliminaries}   \label{sec;Pre}

\subsection{Random Closed Sets}  \label{sec;RSC}
Random closed sets play a key role in many parts of this paper,
particularly in the description of the main limiting objects. This section is an overview of mostly well-known facts about
random closed sets. Unless stated otherwise, these facts are taken
from  \cite{molchanov:2017}.

We work with an underlying space $E$, which will be either 
 $[0,1]$ or $\bbr_+$. We write $\mathcal{G},\mathcal{F}$,
 $\mathcal{F}^\prime$,  $\mathcal{K}$ and $\mathcal{K}^\prime$ for the
 family of open, closed, nonempty closed,  compact and nonempty
 compact sets in $E$, respectively.  If we want to emphasize the choice of $E$, we will use a notation of the type $\mathcal{F}([0,1])$. For any $A\subset E$, we define
\begin{align}
    &\mathcal{F}^A = \{ F \in \mathcal{F}: F \cap A = \emptyset \} ,    \label{eq;miss} \\
    & \mathcal{F}_A  = \{  F \in \mathcal{F}: F \cap A \neq \emptyset \} .    \label{eq;hit}
\end{align}
The \textit{Fell} topology on $\calF$ 
is generated by the sub-basis consisting of    $\{\mathcal{F}_G : G
\in \mathcal{G} \}$ and $\{ \mathcal{F}^K : K \in  \mathcal{K} \}$,
under which the space $\mathcal{F}$ is compact and metrizable. In the
case $E=[0,1]$,  $\mathcal{F}=\mathcal{K}$ and
the Fell topology on $\mathcal{F}$ agrees with the so-called
\textit{myopic} topology on $\mathcal{K}$. In particular,
$\mathcal{F}^\prime$, with the subspace topology, is metrizable by the
Hausdorff metric  
\begin{equation}  \label{eq;HdfMtc}
    \rho_\text{H} (F_1,F_2) : = \max \left\{
    \sup_{x\in F_1} \rho(x, F_2), \sup_{x \in F_2} \rho
    (F_1, x)
    \right\}
    , \quad F_1, F_2 \in \mathcal{F}^\prime,
\end{equation}
where  $\rho(\cdot , \cdot)$ is the standard distance function 
\begin{equation}  \label{eq;psMtc}
    \rho(x,F) = \min \{ \abs{ x  - y } : y \in F   \}.
\end{equation}
We will also use another common  
distance function 
\begin{equation}  \label{eq;ssMtc}
    \rho(F_1,F_2) = \min \{ \abs{x-y}: x\in F_1, \, y\in F_2\}.
\end{equation}

For either choice of $E$, a \textit{random closed set} is a measurable
mapping from a probability space $(\Omega,\mathscr{F},\PP)$ to
$(\mathcal{F},\mathcal{B}(\mathcal{F}))$. A  sequence of random closed
sets $\{F_n\}_{n\in  \bbn}$ weakly converges to $F$ if 
\begin{equation} \label{eq;RSCwkcvg}
    \lim_{n\to \infty}
    \PP \{ F_n \cap A \neq \emptyset \} 
    = \PP \{ F \cap A \neq \emptyset \} , \quad
    \text{for any } A \in \mathcal{A} \cap \mathfrak{S}_F.
\end{equation}
Here, $\mathcal{A}$ is the collection of all finite unions of open intervals, and $\mathfrak{S}_F$ is the
collection of all continuity sets of $F$: 
\begin{equation} \label{eq;F-cont}
    \Big\{ B \in \mathcal{B}: B \text{ is relatively compact and } \PP\{ F \cap \text{cl } B \neq \emptyset \} = \PP \{   F \cap \text{int } B \neq \emptyset \} \Big\}.
\end{equation}

  We now introduce the random closed sets
of our primary concern. 
For a $\beta \in (0,1/2)$, let $\{ Z(t) \}_{t \in \bbr_+}$ be a standard $\beta$-stable subordinator, which is 
an increasing L\'{e}vy process with the Laplace transform 
\begin{equation}  \label{eq;betaSUB}
   \mathbb{E}\exp \left\{ -\theta Z (t)   \right \} 
    =  \exp \left\{ - t \theta^\beta  \right \}, \quad \theta \in \bbr_+. 
\end{equation}
The $\beta$-\textit{stable regenerative set} $R$ is the closure of the range of $\{Z(t)\}_{t\in \bbr_+}$,
\begin{equation} \label{eq;betaSRS}
    R  =  \text{cl} \left\{ Z(t) : t \in \bbr_+ \right \} . 
\end{equation}

Next, take a random variable $Z^\ast(0)$  independent of  $\{ Z(t)
\}_{t \in \bbr_+}$ with the  distribution
\begin{equation} \label{eq;shiftSUB0.0}
    \mathbb{P}\{ Z^\ast(0) \leqslant x \} = x ^{1-\beta}, \quad x \in [0,1].
\end{equation}
We define the process
\begin{equation} \label{eq;shiftSUB0.1}
    Z^\ast(t) = Z^\ast(0) + Z(t), \quad t \in \bbr_+,
\end{equation}
whose range induces another random closed set
\begin{equation} \label{eq;shiftSRS}
    R^\ast  := \text{cl } \{ Z^\ast(t) : t \in \bbr_+ \}= Z^\ast(0) + R.
  \end{equation}

Let $m^\phi$ be the measure associated with the dimension (or gauge)
function
\begin{equation} \label{e:phi}
\phi(x) =  x^\beta (\log \abs{\log x})^{1-\beta}. 
\end{equation}
According  to Theorem 1 in \cite{taylor:wendel:1966}, there is a
finite positive constant $c_\beta$ such that on an event of
probability 1,
\begin{equation} \label{eq;HdfMea}
 m^\phi ( R \cap [0,t] )     = c_\beta Z ^ \leftarrow (t) \ \ \text{for all $t \in \bbr_+$.}
 \end{equation}
Note that $Z^\leftarrow(\cdot)$ is a standard Mittag-Leffler process,
which is self-similar with exponent $\beta$ and has continuous sample
paths. An immediate consequence of \eqref{eq;HdfMea} is that on the
same event of probability 1,
\begin{align} \label{eq;HdfMea*}
& m^\phi ( R^\ast \cap [0,t] ) 
    = c_\beta  Z ^ {\ast \leftarrow }(t) \ \ \text{for all $t \in \bbr_+$.}
\end{align} 

In the sequel, we will be mostly interested in the restriction 
\begin{equation} \label{eq;barSRS}
 \overline{R^\ast}:=R^\ast \cap [0,1]
\end{equation}
of $R^\ast$ to the unit interval. Furthermore, we will need to sample
points from this restriction according to the normalized measure
$m^\phi$ on it. We now set up a technical framework for doing
so. Assuming  the random set $R^\ast$ is defined on some probability
space $(\Omega,\mathscr{F},\PP)$, let $\{U_i\}_{i\geqslant 1}$ be 
 a sequence of $i.i.d.$ standard uniform random variables defined on
 another probability space, say, $(\Omega_1,\mathscr{F}_1,\PP_1)$. For 
$\omega \in \Omega$ define 
\begin{equation}  \label{eq;eta_omega}
    \eta_\omega: [0,1] \to [0,1], \quad
    t \mapsto
    \frac{m^\phi ( R^\ast(\omega) \cap [0,t] )}{m^\phi ( R^\ast(\omega) \cap [0,1] )}
\end{equation}
if the denominator is positive, while setting $\eta_\omega(t)\equiv t$
for the $\omega$ in the zero probability event that the denominator
vanishes. We define  
\begin{equation} \label{eq;barRsi}
    J_i(\omega,\omega_1)  =  \eta^\leftarrow_\omega( U_i(\omega_1) ),
    \    i=1,2,\ldots, \ (\omega,\omega_1) \in (\Omega\times \Omega_1),
  \end{equation}
and view $\bigl(\{ Z^{\ast \leftarrow}(t)\}_{t\in \bbr_+},
\overline{R^\ast}, \{ J_i\}: i \in \bbn \}  \bigr)$ as a random element
of the space $C[0,\infty)\times \mathcal{F}([0,1])\times\bigl(\mathcal{F}([0,1])\bigr)^\infty$,  defined on
the product probability space $\bigl( \Omega\times \Omega_1,
\mathscr{F}\times  \mathscr{F}_1, \PP\times\PP_1\bigr)$. The law of
this random element will be very important in the sequel.  It follows
from \eqref{eq;HdfMea*} that  for $\PP-a.s.$ $\omega \in \Omega$
\begin{equation}\label{eq;barCL0.0}
    \PP_1 \left\{ 
     \cup_{i\geqslant 1} \{  J_i(\omega,\omega_1) \}  \subset
    \overline{R^\ast} 
    \right \} =
    \PP_1 \left\{ 
    \text{cl}  \left( \cup_{i\geqslant 1} \{  J_i(\omega,\omega_1)  \} \right) =
    \overline{R^\ast}    \right \} = 1 ,
\end{equation}
 Therefore also 
 \begin{equation}\label{eq;barCL0.1}
   \PP \left\{ 
     \cup_{i\geqslant 1} \{  J_i(\omega,\omega_1) \}  \subset
    \overline{R^\ast} 
    \right \} =
    \PP \left\{ 
    \text{cl}  \left( \cup_{i\geqslant 1} \{  J_i(\omega,\omega_1)  \} \right) =
    \overline{R^\ast}    \right \} = 1 .   
\end{equation}

\subsection{Null Recurrent Markov Chains  }   \label{sec;MC}

We introduce certain null recurrent Markov chains
from an  ergodic  theoretic perspective. 
More details can be found in \cite{aaronson:1997:surv50}. 

Let $\{Y_t\}_{t \in \mathbb{Z}}$ be an irreducible, aperiodic, and
null recurrent Markov chain on $\mathbb{Z}$. We specify a unique invariant measure $(\pi_i)_{i\in\mathbb{Z}}$ by taking $\pi_0 =1$.
Let $(E,\mathcal{E})$ be the path space $(\bbz^\bbz,
\mathcal{B}(\bbz^\bbz))$ and $\PP_i(\cdot)$  the law $\PP\left\{
  \cdot \, \vert \, Y_0 = i \right\}$ induced by the trajectories of
the Markov chain on $E$. 
Setting 
\begin{align}
 &\mu(\cdot) = \sum_{i\in \bbz} \pi_i \PP_i(\cdot) ,     \label{eq;mu} \\
 &\theta: E \to E, \quad (\cdots, y_0, y_1, y_2,\cdots) \mapsto (\cdots, y_1, y_2, y_3,\cdots),   \label{eq;leftshift}
\end{align}
makes $(E,\mathcal{E},\mu,\theta)$ a 
measure preserving, conservative and ergodic dynamical system, see \cite{harris:robbins:1953}.

We consider the
\textit{first visit time} to $0$, 
\begin{equation}  \label{eq;1STvisit}
    \varphi(y) = \inf \{ t \geqslant 1 : y_t = 0 \}, \quad y \in E, 
\end{equation}
and adopt the following assumption. 
\begin{assumption}  \label{ass;MCmemo}
For some $\beta \in (0,1/2)$ and slowly varying function 
$L$, 
\begin{align}
& \overline{F}(n):= \mathbb{P}_0 \{ \varphi > n \}
    = n^{-\beta} L (n),      \label{eq;MCmemo0.0} \\
& \sup_{n \geqslant 0 }\frac{n \mathbb{P}_0\{ \varphi = n\}}{\overline{F}(n)} < \infty.     \label{eq;MCmemo0.1}
\end{align} 
\end{assumption}

\begin{remark} \label{rmk;Mcmemo}
The main results of the paper require the assumption
(\ref{eq;MCmemo0.1}), see Theorem B in \cite{doney:1997}.
\end{remark}

The \textit{wandering rate sequence} $\{w_n\}_{n \geqslant 0}$ is defined
by
\begin{equation} \label{eq;wander0.1}
 w_n  = \mu \left\{ \cup_{k=0}^n A_k \right \} \ \ \text{with} \  \ A_n
 = \{ y \in E : y_n = 0 \}, \   n \in \bbn_0.
\end{equation}
Under  \eqref{eq;MCmemo0.0}   it follows from Lemma 3.3 in
\cite{resnick:samorodnitsky:xue:2000} that 
\begin{equation}  \label{eq;wander0.2}
    w_n \sim   n^{1-\beta} L(n)  \big/ (1-\beta).
\end{equation}

The extreme value analysis in this paper requires 
certain additional delicate details hidden in
$(E,\mathcal{E},\mu,\theta)$. For each $n\in \bbn_0$, we define a
probability measure on $E$ by 
\begin{equation}  \label{eq;mu_n}
    \mu_n (\cdot)  = \mu \big(\cdot \cap \bigcup_{k=0}^n A_k \big ) \big/ w_n. 
\end{equation}
 Let $\{ Y^{(k;n)} \}_{k\in \bbn_0}$ be $i.i.d.$ random elements in $E$
 with law $\mu_n$. We are interested in the (random) zero sets 
\begin{equation}  \label{eq;Ikn}
    I_{k;n} = \{  0 \leqslant t \leqslant n : Y^{(k;n)}_t = 0 \}, \  k\in \bbn_0,
\end{equation}
and their intersections. For fixed  $n,k \in \bbn$ we define 
\begin{equation}  \label{eq;jk,1;n}
    j_{k,1;n}  = \inf \{ j > k : I_{j;n} \cap I_{k;n}
    \neq \emptyset\} 
\end{equation}
and continue inductively by setting for $i\geq
2$, 
\begin{equation}  \label{eq;jk,i;n}
    j_{k,i;n} = \inf \{ j > j_{k,i-1;n} : I_{j;n} \cap I_{k;n}
    \neq \emptyset\}  
  \end{equation}
if $j_{k,i-1;n}<\infty$ and $ j_{k,i;n}=\infty$ otherwise. For $i\geq
1$, on the event $\{ j_{k,i;n}<\infty \}$ we define  
\begin{equation}  \label{eq;Ik,i;n}
    I_{k,i;n} = I_{k;n} \cap I_{j_{k,i;n};n}.
  \end{equation}

For $k\in\bbn$, consider the random probability
\begin{equation} \label{eq;pk;n}
    \overline{p}_{k;n} : = \PP \left\{
    I_{k;n} \cap I_{0;n} \neq \emptyset \, \big| \, I_{k;n}
    \right \},
  \end{equation}
and note that, conditionally on $I_{k;n}$,  $j_{k,1;n}-k$ is
geometrically distributed with success probability $
\overline{p}_{k;n}$. The following theorem is interesting on its own
right. We precede it with some notation. Let $\bigl(\{
Z_k^{\ast\leftarrow}(t)\}_{t\in \bbr_+}, 
\overline{R_k^\ast}, \{ J_{k,i}\}: i\in \bbn\bigr), \,
k=1,2,\ldots$ be $i.i.d.$ copies of the random element $\bigl(\{
Z^{\ast \leftarrow}(t)\}_{t\in \bbr_+}, 
\overline{R^\ast}, \{ J_i\}: i\in \bbn \bigr)$  
in $C[0,\infty)\times \mathcal{F}([0,1])\times \bigl(\mathcal{F}([0,1])\bigr)^\infty$ constructed in
Section \ref{sec;RSC}. Let $\mathbf{\Gamma}_k, \, k=1,2,\ldots$ be
an independent of them i.i.d. sequence of unit rate Poisson processes on
$(0,\infty)$. That is, each $\mathbf{\Gamma}_k=\{ \Gamma_{k,i} \}_{i\in \bbn}$ consists of the arrival times of a unit rate Poisson processes on
$(0,\infty)$ (listing the points
of $\mathbf{\Gamma}_k$ in the increasing order). 

\begin{theorem} \label{thm;4joint}
Under Assumption \ref{ass;MCmemo} there is a constant $c_\infty \in
(0,1)$ such that  for any $K,m \in \bbn$
\begin{align}\label{eq;4joint0.0}
   & \left( 
    \frac{w_n}{\vartheta_n} \overline{p}_{k;n}, \bigl(
  j_{k,i;n}\overline{p}_{k;n}\bigr)_{1\leqslant i \leqslant m} ,  
   \frac1n  I_{k;n},  \bigl( I_{k,i;n}/n\bigr) _{1\leqslant i \leqslant m}
    \right )_{1\leqslant k \leqslant K} \\
    \Rightarrow 
   & \left( 
    c_\infty  Z_k^{\ast \leftarrow}(1), 
    \bigl( \Gamma_{k,i}\bigr)_{1\leqslant i \leqslant m},
    \overline{R_k^\ast}, 
    \left\{  {J_{k,i}  } \right \}_{1\leqslant i \leqslant m}
    \right)_{ 1\leqslant k \leqslant K} \notag 
\end{align}
as $n\to\infty$, weakly 
in the space  $\left(\mathbb{R}_+^ {1+m}\times \left(\mathcal{F}([0,1])
\right)^{1+m}\right)^K$, where  
\begin{equation} \label{eq;vartheta_n}
    \vartheta_n = \frac{(2-\beta)n^\beta}{\beta L (n)}, \ 
       n \in \bbn. 
\end{equation} 
\end{theorem}
This theorem is proved in Appendix \ref{sec;AppA}, and so is the
following proposition that establishes an 
exponential integrability of $\#I_{1,1;n}$ in both annealed and
quenched situations. 
 
\begin{proposition} \label{prop;1,1card}
\begin{enumerate}[label=(\roman*)]
    \item   Let $c_\infty\in (0,1)$ be the constant in Theorem
      \ref{thm;4joint}. Then 
    \begin{equation} \label{eq;1,1card0.1}
        \PP \left\{ I_{1,1;n} \geqslant m \right \} \leqslant
        (1-c_\infty)^m, \ \ m=1,2,\ldots. 
    \end{equation}
    \item Let $I_{1;n}$ be defined on 
    $(\Omega,\mathscr{F},\PP)$. Then
    for any $\epsilon>0$, there is $\delta=\delta(\epsilon)>0$, $C=C(\epsilon)>0$, and
    an event $\Omega_n \subset \Omega$
    satisfying $ \PP \left\{ \Omega_n \right \} \geqslant  1 -
    \epsilon$ such that 
    \begin{equation} \label{eq;1,1card0.2}
    \sup_{\omega \in \Omega_n}
    \EE_\omega e^{\delta \cdot \# I_{1,1;n}}  \leqslant C ,
    \end{equation}
    where $\EE_\omega$ denotes the conditional expectation $\EE\left\{\cdot \, | \, I_{1,n}(\omega)  \right\}$.
\end{enumerate}
\end{proposition}

The last proposition of this subsection is adapted from
(A.3) and (A.9) of \cite{chen:samorodnitsky:2020}. 
\begin{proposition} \label{prop;ALLcard}
  For any $p>0$ there is $\mu_p<\infty$ such that
  \begin{equation} \label{e:mu.p}
\sup_{n\geqslant 1} \EE\bigl(  \overline{F}(n)\# I_{1;n} \bigr)^p\leq
\mu_p. 
\end{equation}
Further, for any $C>0$ there is $c>0$ so that for all $n\in \bbn$,
    \begin{equation} \label{eq;ALLcard0.0}
    \PP 
    \left\{
    \# I_{1;n} \geqslant \frac{c\log n}{ \overline{F}(n)}
    \right\} \leqslant n^{-C}.
\end{equation}
\end{proposition}

\subsection{Distributions in the Gumbel maximum domain
  of  attraction} \label{subsec:gumbel}
Recall that a  distribution $H$, with an unbounded support on the right, 
is in the Gumbel maximum domain
  of  attraction if and only if there exist
  $x_0 \in \mathbb{R}$ and $c(x)\to c>0$ as $x\to\infty$ such
  that for $x_0< x <  \infty$ 
\begin{equation} \label{eq;VM1}
    \overline{H}(x) = c(x) \exp \left \{
    - \int_{x_0} ^ x \frac{1}{h(u)} du 
    \right \}
\end{equation}
where $h$ (the
so-called auxiliary function) is an absolutely continuous
positive function  on $(x_0,\infty)$ with density $h^\prime$ satisfying
$\lim_{u \to \infty} h^\prime (u) =0$; we refer the reader to
\cite{resnick:1987:EVRVPP} and \cite{goldie:resnick:1988} for more details. 
The function $h$ must satisfy
$h(x)=o(x)$ as $x\to\infty$; if the distribution $H$ is also
subexponential, then its support is unbounded on the right and
$\lim_{u \to \infty} h(u) =\infty$. 

For a distribution
$H$ satisfying \eqref{eq;VM1}, the centering and scaling required  for
convergence in the extremal limit theorem  can be chosen as 
$$
        b_n = \left(  \frac{1}{1-H} \right ) ^ \leftarrow  (n) , \ 
         a_n = h ( b_n )\,.
         $$
We will often use the following fact: if one
replaces the function $c(\cdot)$ in \eqref{eq;VM1} by an asymptotically
equivalent function, and denotes the new normalizing sequences by
$(\tilde a_n)$ and $(\tilde b_n)$, then
\begin{equation} \label{e:small.diff}
  \lim_{n\to\infty} \frac{b_n-\tilde b_n}{a_n}=0, \ 
  \lim_{n \to \infty} \frac{ \tilde a_n } {a_n } = 1 \, .
  \end{equation}

\subsection{Random Sup-Measures}   \label{sec;RSM}
We will deal with sup-measures taking values in
$\overline{\mathbb{R}}=[-\infty, \infty]$. The
main reference is \cite{obrien:torfs:vervaat:1990}. 

A  sup-measure is a mapping $m:\mathcal{G} \to \overline{\mathbb{R}}$ such that $m(\emptyset) =-\infty$ and 
$ m (  \cup_\alpha G_\alpha ) = \vee_\alpha m (G_\alpha) $ for an
arbitrary collection $( G_\alpha )$ of open sets. 
The sup-derivative $d^\vee m$ of $m$ is 
$$
        d^ \vee m (t) = \bigwedge_{t \in G} m (G) ;
$$
it is automatically an upper semicontinuous $\overline{\mathbb{R}}$-valued 
function of $t$. Given any $\overline{\mathbb{R}}$-valued 
function $f$,  the sup-integral of $f$ 
$$
        i^ \vee f (G) = \bigvee_{t \in G } f (t), \quad  G \in
        \mathcal{G} 
$$
is a sup-measure.  The 
 domain of a sup-measure can be extended to all Borel sets via  
$$
        m(B) = \bigvee_{t \in B}  d^\vee (t), \quad B \ \ \text{Borel.}
$$
The collection $\text{SM}$ of sup-measures admits a natural metrizable 
sup-vague topology with its corresponding Borel measurability, which 
allows one to talk about random sup-measures. In particular, if
$\{\mathcal{M}_n\}_{n \geqslant 1}$ and $\mathcal{M}$ are random sup-measures, then 
$\mathcal{M}_n \Rightarrow \mathcal{M}$ if and only if 
\begin{equation} \label{eq;rsmWC1}
    (\mathcal{M}_n (I_1),\ldots,\mathcal{M}_n (I_m) ) \Rightarrow 
    (\mathcal{M} (I_1),\ldots,\mathcal{M} (I_m) )
\end{equation}
for arbitrarily disjoint open intervals $I_1,\ldots,I_m$ such that $\PP\{ \mathcal{M}(I_i) = \mathcal{M}(\overline{I_i}) \} =1$ for all $i=1,\ldots,m$.

A stochastic process $\{X_t\}_{t\in \bbn}$  induces a family of
random sup-measures $\{ \mathcal{M}_n(\cdot) \}_{n\geqslant 1}$ via 
\begin{equation} \label{eq;RSM0.0}
    \mathcal{M}_n(B) : = \max_{t \in n B} X_t, \quad B \in \mathcal{B}(E),
\end{equation} 
We now describe the limiting random sup-measures appearing in our main results. The construction is doubly stochastic.

First, let $\alpha\in (0,1)$ and $\beta\in (0,1/2)$, and denote by $P_\beta$ the law
of the closed range 
$R$ of the $\beta$-stable subordinator in (\ref{eq;betaSRS}).
Consider a Poisson point process $\mathcal{N}$ on $\bbr \times \bbr_+  \times
\mathcal{F}(\bbr_+)$ with mean measure  
$$
    e^{-x} dx \times (1-\beta) y^{-\beta} d y  \times  d P_\beta, 
$$
defined on some probability space $(\Omega_1,\mathscr{F}_1,\PP_1)$, and 
let $(X_k, Y_k, R_k)_{k\in \bbn}$ be a measurable enumeration of its
points. 
For each point $(X_k, Y_k, R_k)=(X_k(\omega_1), Y_k(\omega_1),
R_k(\omega_1))$ the dimension function $\phi$ in \eqref{e:phi} produces
a (random) measure 
\begin{equation} \label{eq;2ndMMea0.2}
    H_k (t,\omega_1) =  m^\phi \left\{ \big( Y_k(\omega_1) +
      R_k(\omega_1) \big) \cap [0,t] \right\},      \quad t \in \bbr_+
\end{equation}
on $\bbr_+$.  The second level of randomness is now
  introduced, conditionally on this point $(X_k(\omega_1),
  Y_k(\omega_1), R_k(\omega_1) )$,  via a Poisson point process
$\mathbb{C}_k=\mathbb{C}_k(\omega_1)$ on $\bbr \times 
\bbr_+$ with the mean measure 
$$
\frac{C_{\alpha,\beta}}{c_\beta} \exp \left\{  - C_{\alpha,\beta}
  \big(  \lambda - X_k (\omega_1)   \big) \right \}  d\lambda  \times  d
H_k(\cdot,\omega_1) ,
$$
 where
\begin{align}
   & C_{\alpha,\beta}: = \left( \frac{1-\beta}{\beta} \right)^{\frac{1}{\alpha} - 1}
     >1 ,  \label{eq;2ndMMea0.1}
\end{align}
and $c_\beta$ is the constant in (\ref{eq;HdfMea}). We assume that the
point processes $(\mathbb{C}_k)$ live on some other probability space 
$(\Omega_2,\mathscr{F}_2,\PP_2)$ and, conditionally on
  $\mathcal N$, are independent of each other. 
 The overall probability space
$(\Omega,\mathscr{F},\PP)$  is
the product space $(\Omega_1\times
\Omega_2,\mathscr{F}_1\times\mathscr{F}_2,\PP_1\times\PP_2)$. 
 
For $k\in\bbn$, let  
$\{ \Lambda_{k,i}, W_{k,i}\}_{i\in\bbn}$ be a measurable
enumeration of the points of  $\mathbb{C}_k$. We note that the
points $\{W_{k,i}\}_{i\in\bbn}$  belong to  the set $Y_k+R_k$ with
probability 1. Consider a random
sup-measure $\mathcal{M}$ on $\bbr_+$ 
given by 
\begin{equation} \label{eq;limRSM0.0}
  \mathcal{M}(B) =   \sup_{\substack{k, i\in \bbn }} \left\{ \Lambda_{k,i} : W_{k,i} \in B \right \}
  , \quad B \in \mathcal{\bbr_+}.
\end{equation}
This sup-measure has certain invariance properties. First of all, it
is clear that $\mathcal{M}$ can be represented in the form
$\mathcal{M}=\varphi\bigl(\mathcal{N}(\omega_1), \omega_2\bigr)$ for
some measurable function $\varphi$, and   the
mean measure of  $\mathcal{N}$ is invariant under positive shifts of
the closed sets. Hence the stationarity of the  $\mathcal{M}$:
$$
\mathcal{M}(r+\cdot) \eid \mathcal{M}(\cdot) \ \ \text{for any} \ \
r\geqslant 0;
$$
cf. Proposition 4.3 in \cite{lacaux:samorodnitsky:2016}. Next,
$\mathcal{M}$ has a self-affiness property: for $a>0$, 
$$
\mathcal{M}(a\cdot) \eid \mathcal{M}(\cdot) + \left(1-\beta +\beta/C_{\alpha,\beta}
\right) \log a, 
$$
which will be established in Remark \ref{rk:prop.M} below. 
  
The random sup-measure $\mathcal{M}$   naturally induces a stochastic
process of independent interest by
\begin{equation} \label{eq;EP0.0}
\mathbb{M}(t) = \mathcal{M}([0,t]), \quad t\in(0,\infty).
\end{equation}
This process is clearly nondecreasing, continuous in probability, and 
the sample paths are in $D(0,\infty) = \cap_{\epsilon >0}
D[\epsilon,\infty)$. The
finite-dimensional distributions can be read off the
following proposition, which implies that $\mathbb{M}(t)
\to -\infty$ as $t\downarrow 0$.

\begin{proposition}  \label{prop;marginEP}
Let $0=t_0<t_1<\cdots<t_k<\infty$. Then for any $x_i\in\bbr, i=1,\ldots,
k$,
\begin{align} \label{e:joint.incr}
&\PP\left( \mathbb{M}\bigl( (t_{i-1},t_i]\bigr)\leqslant x_i, \, i=1,\ldots,
  k\right) =\exp\biggl\{ - \Gamma\bigl( 1-1/C_{\alpha,\beta}\bigr)  \\
 &\left.  \int_0^{\infty} (1-\beta)y^{-\beta}
 \EE\left[ \sum_{i=1}^k e^{-C_{\alpha,\beta} x_i}  \bigl[ Z^\leftarrow \bigl((t_i-y)\bigr)_+- Z^\leftarrow\bigl(
  (t_{i-1}-y)_+\bigr)\bigr]\right]^{1/C_{\alpha,\beta}} dy\right\}. \notag 
\end{align}

In particular, for any $t>0$ and $x\in\bbr$  we have
\begin{equation} \label{eq;marginEP0.0}
\PP\left\{ \mathbb{M}(t) \leqslant x  \right\} = 
\exp\left\{ 
-  K(\alpha,\beta) t^{1-\beta+\beta C_{\alpha,\beta}^{-1}}  
e^{-x}\right\},
\end{equation}
where
$$
K(\alpha,\beta)= (1-\beta) \Gamma\bigl( 1-1/C_{\alpha,\beta}\bigr) 
B\bigl( 1-\beta, 1+\beta/C_{\alpha,\beta}\bigr) \EE\bigl(
Z^\leftarrow(1)\bigr)^{1/C_{\alpha,\beta}}.
$$
Here $\Gamma(\cdot)$ and $B(\cdot,\cdot)$ are, respectively, the Gamma
function and the Beta function.
\end{proposition}

\begin{remark} \label{rk:prop.M}
{\rm 
Leting $t\downarrow 0$ in \eqref{eq;marginEP0.0} shows that $\mathbb{M}(t)
\to -\infty$ as $t\downarrow 0$. Next, let $a>0$. By
\eqref{e:joint.incr} and  the self-similarity of $Z^\leftarrow$,
\begin{align*}
&\PP\left( \mathbb{M}\bigl( (at_{i-1},at_i]\bigr)\leqslant x_i, \, i=1,\ldots,
  k\right) \\ 
  =&\exp\biggl\{ - a^{\beta/C_{\alpha,\beta}}\Gamma\bigl( 1-1/C_{\alpha,\beta}\bigr)  
     \int_0^{\infty} (1-\beta)y^{-\beta} \\
   &\hskip 0.35in \EE_1 \left[  \sum_{i=1}^k e^{-C_{\alpha,\beta} x_i} \bigl[Z^\leftarrow \bigl((t_i-y/a)\bigr)_+- Z^\leftarrow\bigl(
   (t_{i-1}-y/a)_+\bigr)\bigr]\right]^{1/C_{\alpha,\beta}} dy  \biggr \}  \\
 =&\exp\biggl\{ - a^{1-\beta+\beta/C_{\alpha,\beta}}\Gamma\bigl( 1-1/C_{\alpha,\beta}\bigr)  
    \int_0^{\infty} (1-\beta)y^{-\beta} \\
  &\hskip 0.35in \EE_1 \left[  \sum_{i=1}^k e^{-C_{\alpha,\beta} x_i} \bigl[Z^\leftarrow \bigl((t_i-y)\bigr)_+- Z^\leftarrow\bigl(
   (t_{i-1}-y)_+\bigr)\bigr]\right]^{1/C_{\alpha,\beta}} dy \biggr\} \\
 =&\PP\left( \mathbb{M}\bigl( (t_{i-1},t_i]\bigr)+ \left(1-\beta +\beta/C_{\alpha,\beta}
\right) \log a\leqslant x_i, \, i=1,\ldots,  k\right). 
\end{align*}
Hence the self-affiness of $ \mathbb{M}$.
}
\end{remark}

\begin{proof}[Proof of Proposition \ref{prop;marginEP}]
We view the collection $\{ \Lambda_{k,i}, W_{k,i}\}_{k,i\in\bbn}$ is a
function of the points of the Poisson process $\mathcal{N}$  and their
$i.i.d.$ marks defined on $(\Omega_2,\mathscr{F}_2,\PP_2)$. Since a
marked Poisson process is still Poisson, and the event described in
the left hand side of \eqref{e:joint.incr} is the event that the
marked Poisson process has no points in a part of its domain, we see
that
\begin{align*}  
&\PP\left( \mathbb{M}\bigl( (t_{i-1},t_i]\bigr)\leqslant x_i, \, i=1,\ldots,
                   k\right) \\
  =&\exp\biggl\{ -  \int_{-\infty}^\infty e^{-z}dz \int_0^{\infty}
                   (1-\beta)y^{-\beta}dy \\  
& \hskip 0.35in \EE_1\bigl[ \PP_2\bigl\{\mathbb{C}_{z,y}\bigl(\cup_{i=1}^k
                    \bigl( x_i,\infty)\times
     (t_{i-1},t_i]\bigr)>0\bigr\}\bigr]\biggr\}, 
\end{align*}
where, given a defined on $(\Omega_1,\mathscr{F}_1,\PP_1)$
$\beta$-stable regenerative set $R$, the process $\mathbb{C}_{z,y}$ is
a defined on $(\Omega_2,\mathscr{F}_2,\PP_2)$ Poisson process on
$\bbr\times\bbr_+$ with mean measure 
$$
\frac{C_{\alpha,\beta}}{c_\beta} \exp \left\{  - C_{\alpha,\beta}
  \big(  \lambda - z   \big) \right \}  d\lambda  \times
d m^\phi \left\{ \big( y +
  R  \big) \cap [0,\cdot] \right\}.
$$
Therefore,
\begin{align*}
& \PP_2\bigl\{\mathbb{C}_{z,y}\bigl(\cup_{i=1}^k
                    \bigl( x_i,\infty)\times
     (t_{i-1},t_i]\bigr)>0\bigr\}\\
=&1- \PP_2\bigl\{\mathbb{C}_{z,y}\bigl(\cup_{i=1}^k
                    \bigl( x_i,\infty)\times
     (t_{i-1},t_i]\bigr)=0\bigr\}
\\
  =&1-\exp\left\{ -c_\beta^{-1}\sum_{i=1}^k e^{-C_{\alpha,\beta}(x_i-z)}
 m^\phi \Bigl\{ \big( y +
  R  \big) \cap (t_{i-1},t_i] \Bigr\}\right\},
\end{align*}
and we conclude by \eqref{eq;HdfMea} that
\begin{align*}
&\EE_1\bigl[  \PP_2\bigl\{\mathbb{C}_{z,y}\bigl(\cup_{i=1}^k
                    \bigl( x_i,\infty)\times
     (t_{i-1},t_i]\bigr)>0\bigr\} \\
  =&\EE_1\left( 1-\exp\left\{
     -\sum_{i=1}^ke^{-C_{\alpha,\beta}(x_i-z)}\bigl[Z^\leftarrow\bigl((t_i-y)\bigr)_+-
     Z^\leftarrow\bigl( 
   (t_{i-1}-y)_+\bigr)\bigr]\right\}\right),
\end{align*}
and \eqref{e:joint.incr} follows by simple integration. Finally, using
\eqref{e:joint.incr} with $k=1$ and $t_1=t,x_1=x$ we obtain 
\begin{align*}
\PP\left\{ \mathbb{M}(t) \leqslant x  \right\} = 
\exp\left\{ - \Gamma\bigl( 1-1/C_{\alpha,\beta}\bigr)
  e^{-x}\int_0^{t} (1-\beta)y^{-\beta} \EE  Z^\leftarrow (t-y) ^{1/C_{\alpha,\beta}} dy\right\},
\end{align*}
and \eqref{eq;marginEP0.0} follows by the self-similarity of
$Z^\leftarrow$ and simple integration. 
\end{proof}

We can obtain an explicit representation of  the restriction of the
sup-measure $\mathcal{M}$ to $[0,1]$.

First, the
restriction of the Poisson point process $(X_k, Y_k+ R_k)_{k\in \bbn}$
to $\bbr \times  \mathcal{K}^\prime([0,1])$ (we only need to look at
nonempty compact sets) can be represented as a Poisson
point process $\mathcal{N}_0$ on $\bbr$ with the mean measure $e^{-x}
dx$, marked by $i.i.d.$ copies of the random closed
set $\overline{R^\ast}$ in \eqref{eq;barSRS}. The markings are independent of
$\mathcal{N}_0$.  
The $k^\text{th}$ copy
$\overline{R_k^\ast}$  is associated  via \eqref{eq;HdfMea*} 
with a shifted stable
subordinator $(Z^*_k)$ satisfying for $t\in [0,1]$,
$$
m^\phi ( \overline{R_k^\ast} \cap [0,t] ) 
    = c_\beta Z_k^ {\ast \leftarrow }(t).
$$
Furthermore, the process 
$\mathcal{N}_0$ itself can be represented by the points
$-\log\Gamma_k, \, k=1,2,\ldots$, where $(\Gamma_k)$ form a standard
unit rate Poisson process on $\bbr_+$. 

Second, for a fixed $k$ the mean
measure of the Poisson point process
$\mathbb{C}_k$ can be rewritten in the form
\begin{align*}
&\frac{C_{\alpha,\beta}}{c_\beta} m^\phi ( \overline{R_k^\ast} \cap
  [0,1] )
  \exp \left\{  - C_{\alpha,\beta}
  \big(  \lambda - X_k (\omega_1)   \big) \right \}   d\lambda
                 \times  d\eta_{k}(\cdot,\omega_1)  \\
 =&C_{\alpha,\beta} Z_k^ {\ast \leftarrow }(1)  \exp \left\{  - C_{\alpha,\beta}
  \big(  \lambda - X_k (\omega_1)   \big) \right \}   d\lambda
                 \times  d\eta_{k}(\cdot,\omega_1),
\end{align*}
where $\eta_{k}(\cdot,\omega_1)$ is the ($\omega_1$-dependent)
probability measure \eqref{eq;eta_omega} associated with $(Z^*_k)$.
Therefore, we can choose a measurable
enumeration of the points of  $\mathbb{C}_k$ by first selecting a
measurable enumeration $\{ \Lambda_{k,i}\}_{i\in\bbn}$ of the points
of the Poisson point process on $\bbr$ with the mean measure
$$
C_{\alpha,\beta} Z_k^ {\ast \leftarrow }(1)  \exp \left\{  - C_{\alpha,\beta}
  \big(  \lambda - X_k (\omega_1)   \big) \right \}   d\lambda
$$
and then attaching to these points independent $i.i.d.$ marks with the
common law $\eta_{k}(\cdot,\omega_1)$. Since the former Poisson
process is, once again, easily generated as a transformation of a unit
rate Poisson process on $\bbr_+$ (say, $(\Gamma_{k,i})$), we conclude
that we can choose a measurable 
enumeration of the points of  $\mathbb{C}_k$ so that, in law, 
\begin{equation} \label{eq;limRSM[0,1]}
  ( \Lambda_{k,i}, W_{k,i}   ) _ {k,i \in \bbn} 
  =
  \left(
   - \log \Gamma_k + \frac{1}{C_{\alpha,\beta} } \Big( -
  \log \Gamma_{k,i} + \log  Z_k^{\ast\leftarrow}(1) \Big),  J_{k,i}
  \right)_{k,i\in \bbn},
\end{equation}
where $\{ \Gamma_k \}_{k\in \bbn }$ is  a unit rate Poisson process on
$\bbr$ independent of the rest random elements that are defined  in
Theorem  \ref{thm;4joint}.

\section{Extremal Limit Theorems for Stationary Semi-Exponential Processes}  \label{sec;SemiExpPro}

We focus on stationary infinitely divisible processes of form
\begin{equation} \label{eq;SEP0.0}
    X_t = \int_E 1_{A_0} \circ \theta ^ t (x) M (dx), \quad t\in \mathbb{Z},
\end{equation}
where $(E, \mathcal{E}, \mu,\theta)$ is the dynamical system described
in \eqref{eq;mu}-(\ref{eq;leftshift}), $A_0=\{y\in E:\, y_0=0\}$, and
$M$ is an infinitely divisible random measure on $(E,\mathcal{E})$
with control measure $\mu$ and with constant local characteristic
triple $(\sigma^2, \nu, b)$; see \cite{samorodnitsky:2016:SPLRD}. We note
that the choice of the indicator function in \eqref{eq;SEP0.0} is
mainly for convenience, and more general integrands could be
considered.

 The processes we consider have the marginal tails that are both
subexponential and in the Gumbel maximum domain of attraction. They
will also have a  certain semi-exponential decay, 
Correspondingly, we choose the 
auxiliary function $h$ in \eqref{eq;VM1} to be of a specific type. The
assumptions are imposed through the local
L\'evy measure $\nu$ of the infinitely divisible random measure.

\begin{assumption} \label{ass;locLevyMea}
For some $\alpha \in (0,1)$ and
$\gamma>0$,
\begin{align}
\overline{\nu}(x):=\nu\bigl( (x,\infty)\bigr)\sim \gamma\overline{H}(x)
  :=\gamma \exp \left\{ - \int_1^ x  
 \frac{du}{u^{1-\alpha} L_\alpha(u)} \right \}, \ x\to\infty \label{eq;LevyMea0.0}
\end{align}
for a differentiable slowly varying function 
 $L_\alpha$  such that  
\begin{align}
& L^\prime_\alpha (x) = o \left(  L_\alpha(x)/x \right). 
 \label{eq;LevyMea0.1}    
\end{align}
\end{assumption}

For large values of the argument both $\overline{\nu}$ and
$\overline{H}$ can be viewed as distributional tails. 
Automatically, these distributions are in the Gumbel maximum domain
  of  attraction. Additionally, by Theorem 2 in \cite{pitman:1980},
  these distributions are also subexponential. We conclude that 
  the distribution of $X_0 \in D(\Lambda) \cap \mathcal{S}$ since $\PP\{ X_0 > x \} \sim \overline{\nu}(x)$, see
  Theorem 1 in \cite{embrechts:goldie:veraverbeke:1979}.

To state the main extremal limit theorems we define two functions by 
\begin{align} 
    & V(x) = \left( 1 \big/ \overline{\nu}  \right) ^\leftarrow (x), \quad
     h(x)  = x ^{1-\alpha} L_\alpha (x), \quad x \geqslant 1.
    \label{eq;aux-h} 
\end{align}
Some properties of these and other important functions are described
in Proposition \ref{prop;VanGogh} in the Appendix B.

The  normalizing constants in the extremal limit theorems are given by 
\begin{equation}  \label{eq;normalising0.0}
     b_n  = V(w_n) + V ( c_\infty \vartheta_n ) , \quad
     a_n  = h \circ V(w_n) , 
\end{equation}
    where   $\{w_n\}_{n\geqslant 1}$ is the wandering rate sequence in
    (\ref{eq;wander0.1}), $\{\vartheta_n\}_{n\geqslant 1}$ is the
    sequence in (\ref{eq;vartheta_n}), and $c_\infty$ is the constant
    in Theorem \ref{thm;4joint}, given explicitly in
    \eqref{eq;c_infty}.

The first extremal limit theorem establishes convergence in the space
of random sup-measures. 
\begin{theorem}  \label{thm;RSMcv}
Let $\{X_t \}_{t\in \bbz}$ be the stationary infinitely divisible
process defined by (\ref{eq;SEP0.0}), such that \eqref{eq;LevyMea0.0}
and \eqref{eq;LevyMea0.1} hold. Assume also that the dynamical system
$(E, \mathcal{E}, \mu,\theta)$  satisfies the 
Assumptions \ref{ass;MCmemo}. Then the random sup-measures
defined  by (\ref{eq;RSM0.0}) satisfy  
\begin{equation} \label{eq;RSMcv0.0}
    \frac{\mathcal{M}_n(\cdot) - b_n}{a_n}
    \Rightarrow \mathcal{M}(\cdot) \quad \text{in } \text{SM}([0,\infty)),
\end{equation}
where $\mathcal{M}$ is given in (\ref{eq;limRSM0.0}).
\end{theorem}

The second extremal limit theorem establishes convergence in the
functional space $D(0,\infty)$. 
\begin{theorem} \label{thm;EPcv}
Under the assumptions of Theorem \ref{thm;RSMcv}, let $ \mathbb{M}_n
(t) = \max_{i\leqslant nt} X_i$, $t\in\bbr_+,\, n=1,2,\ldots$.  Then
\begin{equation} \label{eq;EPcv0.1}
    \left\{ \frac{ \mathbb{M}_n(t) - b_n}{a_n}\right\}_{t>0} 
    \Rightarrow \left\{ \mathbb{M}(t) \right\}_{t>0} \quad \text{in } 
    \big( D(0,\infty), J_1 \big), 
\end{equation}
where $\{ \mathbb{M}(t)  \}_{t>0}$ is the stochastic  process in 
(\ref{eq;EP0.0}). 
\end{theorem}
\begin{remark} \label{rk:not.time.change}{\rm 
It follows from \eqref{eq;marginEP0.0} that the limiting process in
Theorem \ref{thm;EPcv} has one dimensional marginal distributions
equal to the one dimensional marginal distributions of the process 
$\bigl( X_G\bigl( t^{1-\beta+\beta/C_{\alpha,\beta}}\bigr), \,
t\geq 0\bigr)$, where $(X_G(t),\, t\geq 0)$ is a shifted Gumbel extremal
process, i.e. a nondecreasing process satisfying 
\begin{equation} \label{e:extr.Gumbel}
\PP\bigl( X_G(t_i)\leq x_i, \, i=1,\ldots, k\bigr) =
\exp\left\{ -K(\alpha,\beta)\sum_{i=1}^k (t_i-t_{i-1})e^{-x_i}\right\}
\end{equation}
for $0=t_0<t_1<\cdots <t_k$ and $x_1\leq x_2\leq \cdots\leq x_k$.

We recall that in \cite{chen:samorodnitsky:2020}, similar extremal limit theorems
are proved for stationary 
processes with marginal tails heavier than the ones in
Theorem \ref{thm;EPcv}. The limiting processes therein are, in distribution,
the power time changes of the standard Gumbel extremal process. 
Analogous time-changed results for Fr\'echet extremal processes are showed in 
\cite{owada:samorodnitsky:2015a} and \cite{lacaux:samorodnitsky:2016}.

Interestingly, the law of the limiting process in Theorem
\ref{thm;EPcv} 
is different from the law of the power time change of
the Gumbel extremal 
process in \eqref{e:extr.Gumbel}. Indeed, if the two processes had the
same law, we would have by \eqref{e:joint.incr} and
\eqref{eq;marginEP0.0}, for $0<t_1<t_2$ and $x_1\leq x_2$,
\begin{align*}
&\exp\left\{ -K(\alpha,\beta) \left[
  t_1^{1-\beta+\beta/C_{\alpha,\beta}}e^{-x_1}+
  (t_2^{1-\beta+\beta/C_{\alpha,\beta}}-t_1^{1-\beta+\beta/C_{\alpha,\beta}})
  e^{-x_2}\right]\right\} \\
 =& \PP\left(  X_G\bigl(
    t_1^{1-\beta+\beta/C_{\alpha,\beta}}\bigr)\leq x_1,
 X_G\bigl(
    t_2^{1-\beta+\beta/C_{\alpha,\beta}}\bigr)\leq x_2\right)   \\
  =& \PP\left( M(t_1)\leq x_1, \, M\bigl( (t_1,t_2]\leq x_2\right) \\
  =& \exp\biggl\{ - \Gamma\bigl( 1-1/C_{\alpha,\beta}\bigr)
      \int_0^{\infty} (1-\beta)y^{-\beta}dy \, \EE
     \biggl[e^{-C_{\alpha,\beta} x_1}  Z^\leftarrow \bigl((t_1-y)\bigr)_+
      \\
  & + e^{-C_{\alpha,\beta} x_2}  \left[Z^\leftarrow \bigl((t_2-y)\bigr)_+
     -Z^\leftarrow (t_1-y)_+ \right]\biggr]^{1/C_{\alpha,\beta}}  
\biggr\},
\end{align*}
and due to the connection between the constants in
\eqref{e:joint.incr} and \eqref{eq;marginEP0.0} this reduces to 
\begin{align*}
&\EE \int_0^{\infty} (1-\beta)y^{-\beta}dy \left( e^{-C_{\alpha,\beta}
  x_1}  Z^\leftarrow \bigl((t_1-y)\bigr)_+\right)
  ^{1/C_{\alpha,\beta}} \\
+& \EE \int_0^{\infty} (1-\beta)y^{-\beta}dy \left[ \left( e^{-C_{\alpha,\beta}
  x_2}  Z^\leftarrow \bigl((t_2-y)\bigr)_+\right)
   ^{1/C_{\alpha,\beta}}  \right.\\
&\left. \hskip 1.3 in 
-\left( e^{-C_{\alpha,\beta}
  x_2}  Z^\leftarrow \bigl((t_1-y)\bigr)_+\right)
                                      ^{1/C_{\alpha,\beta}} \right] \\
=&\EE \int_0^{\infty} (1-\beta)y^{-\beta}dy \biggl[  e^{-C_{\alpha,\beta}
   x_1}  Z^\leftarrow \bigl((t_1-y)\bigr)_+ \\
 & \hskip 1.3in  +e^{-C_{\alpha,\beta}
  x_2}  \bigr[Z^\leftarrow \bigl((t_2-y)\bigr)_+      -Z^\leftarrow
  (t_1-y)_+ \bigr]  \biggr]^{1/C_{\alpha,\beta}}.
\end{align*}
We argue that this is impossible since $C_{\alpha,\beta}>1$. Indeed,
for every $t_1>0$ and $t_2,t_3>t_1$, a simple convexity argument shows
that for $C>1$ we have
\begin{equation} \label{e:convexity}
t_2^C-t_1^C+t_3^C< (t_2-t_1+t_3)^C.
\end{equation}
Now apply \eqref{e:convexity} with
\begin{align*}
&t_1=e^{-x_2} \bigl[Z^\leftarrow
\bigl((t_1-y)\bigr)_+\bigr]^{1/C_{\alpha,\beta}}, \
t_2=e^{-x_1} \bigl[Z^\leftarrow
\bigl((t_1-y)\bigr)_+\bigr]^{1/C_{\alpha,\beta}}, \\
&t_3=e^{-x_2} \bigl[Z^\leftarrow
\bigl((t_2-y)\bigr)_+\bigr]^{1/C_{\alpha,\beta}}.
\end{align*}
}

\end{remark}

\medskip

\begin{remark}{\rm 
We will prove both theorems with the time domain restricted to the
interval $[0,1]$. The general case is only notationally different.
}
\end{remark}

As it is often done when analyzing the extremes of subexponential processes,  we start by decomposing
the process in \eqref{eq;SEP0.0} into a  sum of two independent processes.  One will
collect the large Poissonian contributions of the original process and the
other will collect the small such contributions. Note that, by
\eqref{eq;LevyMea0.1}, we can choose $x_0>0$ (which we assume
to be 1 for notational simplicity)  satisfying 
    \begin{equation} \label{eq;x0}
    \left( \frac{1}{x^{1-\alpha} L_\alpha (x)} \right)^\prime 
    < 0 \quad \text{for all } \, x > x_0, 
  \end{equation}
and  we split the random measure $M$ in (\ref{eq;SEP0.0}) into a  sum 
$M\eid M^{(1)} + M^{(2)}$
of two  independent infinitely divisible random measures $M^{(1)}$ and
$M^{(2)}$  with the same control measure as $M$ and 
with  constant local characteristic 
$\left(0,[\nu]_{(x_0,\infty)},0
    \right)$ and $\left(\sigma^2,[\nu]_{(-\infty,x_0]},b
    \right)$, respectively. We define  two independent stationary
    infinitely divisible processes by 
\begin{equation} \label{eq;SEP0.1}
    X^{(i)}_t    = \int_E 1_{A_0} \circ \theta ^t (x)  M^{(i)}(d x ),   
    \quad t \in \bbz, , i =1,2. 
\end{equation} 
This gives us a desired decomposition 
\begin{equation} \label{eq;SEP0.2}
    \{X_t\}_{t\in \bbz} \eid   \{ X^{(1)}  _ t \}_{t\in \bbz}  + 
  \{X^{(2)} _t \}_{t\in \bbz}.
\end{equation}
By construction, the random variables $\{ X^{(1)}_t \}_{t\in \bbz}$ are
compound Possion. For each $n\in\bbn$, it is convenient to take a series representation of
$\{ X^{(1)}_t \}_ {0\leqslant t \leqslant n}$, which arranges the
Poissonian jumps in the decreasing order. The
representation uses crucially the zero sets $(I_{k;n})$ defined
in \eqref{eq;Ikn}.  It follows 
from Corollary 3.4.2 in \cite{samorodnitsky:2016:SPLRD} (see also (4.12) in
\cite{chen:samorodnitsky:2020}) that
\begin{equation} \label{eq;seriesREP}
\left( X^{(1)}_t  \right)  _ {0\leqslant t \leqslant n}
\eid \left(  \sum_{j\geqslant 1} V_1\left( w_n/\Gamma_j
  \right) 1_{\{ t \in I_{j;n} \} }     \right)_{_ {0\leqslant t
    \leqslant n}}.
\end{equation}
Here $(\Gamma_j)$ are the ordered arrival times of a unit rate Poisson
process on $\bbr_+$ independent of the $i.i.d.$ zero sets
$(I_{j;n})$. Furthermore, $V_1$ is a truncated function $V$ in
\eqref{eq;aux-h}: 
$$
    V_1 (y): = 
    \begin{cases}
    V(y), \; &\text{if } y > 1 \big/ \overline{\nu}(x_0) \\
    0,  &\text{otherwise}
    \end{cases}  . 
$$
For notational simplicity,  we will hereafter use \eqref{eq;seriesREP} with $V$ instead of
$V_1$. We keep in mind that this function  vanishes in a neighborhood of $0$.

In the sequel, we will view $\{ X^{(1)}_t \}_{t\in \bbz}$ as
defined by the series in \eqref{eq;seriesREP} and write 
\begin{equation} \label{eq;rsmX0.0}
    \mathcal{M}_n (B)  = 
    \max_{ t \in nB} \left\{
    X^{(2)}_t  +
    \sum_{j\geqslant 1} V \left( w_n/\Gamma_j \right) 1_{\{ t \in I_{j;n} \} }  
    \right \}, \quad B \in \mathcal{B}([0,1]).
  \end{equation}
The proofs of Theorems \ref{thm;RSMcv} and \ref{thm;EPcv} 
  use a number
of random sup-measures related to \eqref{eq;rsmX0.0} and we list them
below. They use the random sets defined in \eqref{eq;Ik,i;n}.
We also use the random sets 
\begin{align} \label{e:hat.I}
&\widehat I_{k;n}= I_{k;n}\setminus \cup_{j=1}^{k-1} I_{j;n}, \ k\geqslant
    1, \\
 &\widehat I_{k,i;n}= I_{k,i;n}\setminus \cup_{j=1}^{i-1} I_{k,j;n}, \ k,i\geqslant
    1. \notag 
\end{align}

For 
$B\subset [0,1]$  and $n\in \bbn$, we define  for $k,i,K \in \bbn$, 
\begin{align}
& \mathcal{M}_{k,i;n} (B)  =  
    \max_{ t \in  \widehat{I}_{k,i;n} \cap n B }
    \left\{  X ^{(2)} _t + 
    \sum_{j\geqslant 1} V \left( w_n/\Gamma_j \right) 1_{\{ t \in I_{j;n} \} }    
       \right \} ; 
                \label{eq;Mki;n}  \\
&\mathcal{M}_{k;n} (B) =  \max_{ t \in  \widehat{I}_{k;n} \cap n B }
    \left\{  X_t^{(2)} + 
    \sum_{j\geqslant 1} V\left( w_n/\Gamma_j \right) 1_{\{ t \in I_{j;n} \} }   
     \right \};
    \label{eq;Mk;n}  \\   
& \mathcal{M}_{[K];n} (B)  = \bigvee_{k=1} ^ K 
  \mathcal{M}_{k;n} (B).
   \label{eq;M[K];n}  
\end{align}
Furthermore, referring to the random sup-measure $\mathcal{M}$ in \eqref{eq;limRSM0.0}
we also define  
\begin{align} 
    & \mathcal{M}_{k,i} (B) =
    \begin{cases}
     \Lambda_{k,i}  , \quad & 
     W_{k,i}\in B \\
     - \infty ,  &  W_{k,i} \not \in  B
    \end{cases} , \quad k,i\in \bbn;
    \label{eq;Mki} \\
    &\mathcal{M}_{k} (B) =
    \bigvee_{i=1}^{\infty} \mathcal{M}_{k,i}(B) ,
    \quad k \in \bbn;
    \label{eq;Mk} \\
    &  \mathcal{M}_{[K]} (B) = \bigvee_{k=1} ^ K \mathcal{M}_{k} (B),
      \quad K \in \bbn.
    \label{eq;M[K]}
\end{align}

\bigskip

The proofs of Theorems \ref{thm;RSMcv} and \ref{thm;EPcv} rely heavily
on Theorem \ref{thm;4joint}. By the Skorohod embedding we may assume
that all random elements appearing in \eqref{eq;rsmX0.0},
\eqref{eq;Mki;n},  \eqref{eq;Mk;n}, 
and \eqref{eq;M[K];n} are defined on 
a common probability space $(\Omega,\mathscr{F},\PP)$ and the
convergence in Theorem \ref{thm;4joint} holds as the $a.s.$ convergence
for these random elements. Furthermore, the random elements appearing
as the limit in the right hand side of \eqref{eq;4joint0.0} are used
to construct the points of the point processes $(\mathbb{C}_k)$ via
\eqref{eq;limRSM[0,1]}. The random variables $(\Gamma_j)$ are already
naturally coupled via \eqref{eq;rsmX0.0} and \eqref{eq;limRSM[0,1]}. 
This way the random sup-measure $\mathcal{M}$
is coupled to the random sup-measures $(\mathcal{M}_n)$. This setup
will be in force  for the
duration of this section, and it follows from
\eqref{eq;rsmWC1} that the following proposition suffices to prove
Theorem \ref{thm;RSMcv}. 
\begin{proposition} \label{prop;rsmPcvg}
For each open interval $B \subset [0,1]$, 
\begin{equation} \label{eq;rsmPcvg0.0}
    \frac{\mathcal{M}_n(B) - b_n}{a_n } 
    \stackrel{P}{\longrightarrow} \mathcal{M}(B). 
\end{equation}
\end{proposition}

This proposition is a consequence of the   three statements below. 
\begin{proposition} \label{prop;top}
For each $k,i \in \bbn$  and each  open interval  $B \subset [0,1]$, 
\begin{equation} \label{eq;top0.0}
    \frac{\mathcal{M}_{k,i;n}(B) - b_n}{a_n} 
    \stackrel{P}{\longrightarrow} \mathcal{M}_{k,i}(B).
\end{equation}
\end{proposition}

In the next two propositions, $B$ is an  open subset of $[0,1]$ and 
we use the notation $\Omega_B =
\{\overline{R^\ast_1} \cap B \neq \emptyset \}$, 
\begin{proposition} \label{prop;middle}
 Let $\ell_n := \lfloor \rho \log n\rfloor$ with $\rho>0$. If $\rho$
 is small enough, then 
    \begin{equation}\label{eq;1n-mid0.0}
     \lim_{\ell_0 \to \infty}
     \liminf_{n \to \infty}
     \PP \left\{
    \mathcal{M}_{1;n} (B)  > \bigvee_{i=2^{\ell_0} }^{2^{\ell_n} - 1} 
    \mathcal{M}_{1,i;n} ([0,1])
     \, \bigg| \,
     \Omega_B
     \right\} = 1 
    \end{equation}
and
    \begin{equation}\label{eq;n-mid0.0}
     \lim_{\ell_0 \to \infty}
     \liminf_{n \to \infty}
     \PP \left\{
     \mathcal{M}_{n} (B)  > \bigvee_{k=2^{\ell_0} }^{2^{\ell_n} - 1} 
     \mathcal{M}_{k;n} ([0,1])
     \right\} = 1 .
    \end{equation}
\end{proposition}

\begin{proposition} \label{prop;bottom}
For any $\rho>0$, we have
    \begin{equation} \label{eq;n-btm0.0}
    \lim_{n\to \infty}
    \PP 
    \bigg\{
    \mathcal{M}_n(B) > \bigvee_{k\geqslant n^\rho} \mathcal{M}_{k;n} ([0,1])
    \bigg\} = 1 
    \end{equation}
 and 
    \begin{equation} \label{eq;1-btm0.0}
    \lim_{n\to \infty}
    \PP 
    \bigg\{
    \mathcal{M}_{1;n}(B) > \bigvee_{i\geqslant n^\rho}
    \mathcal{M}_{1,i;n} ([0,1]) \, \bigg| \, 
    \Omega_B     \bigg\} = 1     .
    \end{equation}
\end{proposition}

We start by showing how Proposition \ref{prop;rsmPcvg} follows from
the three statements above.

\begin{proof}[Proof of Proposition \ref{prop;rsmPcvg}]
First, we note that by \eqref{eq;1n-mid0.0} and \eqref{eq;1-btm0.0} 
$$
    \lim_{ m\to \infty }
  \liminf_{n \to \infty} 
  \mathbb{P} 
  \left\{ \mathcal{M}_{1;n} (B) = \bigvee_{i=1}^m \mathcal{M}_{1,i;n} (B)
  \, \Big|  \,  \Omega_B
   \right \} = 1 .
$$

Second, for almost every $\omega\in\Omega_B^c$, we also have
$\overline{R^\ast_1} \cap \overline{B}=\emptyset$ because the stable
regenerative set does not hit fixed points. Due to
$I_{1;n}/n\to \overline{R^\ast_1}$ $a.s.$, we therefore see that for almost every
$\omega\in\Omega_B^c$, $I_{1;n}\cap
nB=\emptyset$ for all $n$ large enough, and hence,
$\mathcal{M}_{1;n} (B)  =  \mathcal{M}_{1,i;n} (B)   = - \infty$  
for all $i\in \bbn$. We deduce that 
$$
  \lim_{m\to \infty }
  \liminf_{n \to \infty} 
  \mathbb{P} 
  \left\{ \mathcal{M}_{1;n} (B) = \bigvee_{i=1}^m \mathcal{M}_{1,i;n} (B)
   \right \} = 1.
   $$
This identity obviously remains valid if $\mathcal{M}_{1;n}$ and $\mathcal{M}_{1,i;n}$ are replaced by
$\mathcal{M}_{k;n}$ and $\mathcal{M}_{k,i;n}$ respectively. 
Thus for any $K \in \bbn$,
$$
  \lim_{m\to \infty }
  \liminf_{n \to \infty} 
  \mathbb{P} 
  \left\{ \mathcal{M}_{k;n} (B) = \bigvee_{i=1}^m \mathcal{M}_{k,i;n} (B), 1\leqslant k \leqslant K
   \right \} = 1. 
$$

We proceed to note from \eqref{eq;n-mid0.0} and
\eqref{eq;n-btm0.0}  that 
$$
    \lim_{ K\to \infty }
  \liminf_{n \to \infty} 
  \mathbb{P} 
  \left\{ \mathcal{M}_n(B) = \mathcal{M}_{[K];n} (B)
  \right \} = 1, 
$$
and conclude that 
\begin{equation} \label{eq;rsmPcvg2.1}
 \lim _{\substack{K\to \infty \\ m \to \infty}}
 \liminf_{n \to \infty} 
  \mathbb{P} 
  \Big\{ \bigvee_{\substack{1\leqslant k \leqslant K \\ 1\leqslant i \leqslant m}} \mathcal{M}_{k,i;n}(B) =
  \mathcal{M}_n (B)
  \Big \} = 1 .
\end{equation}

Finally, we note that
\begin{equation} \label{eq;rsmPcvg2.2}
\lim  _{\substack{K\to \infty \\ m \to \infty}}
\bigvee_{\substack{1\leqslant k \leqslant K \\ 1\leqslant i \leqslant
    m}} \mathcal{M}_{k,i}(B) =   \mathcal{M} (B) \ \ a.s..
\end{equation}
By the  standard
``convergent together" argument, (\ref{eq;rsmPcvg0.0}) follows from Proposition \ref{prop;top},
\eqref{eq;rsmPcvg2.1} and \eqref{eq;rsmPcvg2.2}. 
\end{proof}

The proof of Theorem \ref{thm;RSMcv} is, therefore, complete apart
from proving Propositions \ref{prop;top}, \ref{prop;middle} and
\ref{prop;bottom} which we now commence.

\begin{proof}[Proof of Proposition \ref{prop;top}]
 We consider $k=1$. Recall that for any $i$, a.s.,
$$
\frac{1}{n}(I_{1,1;n}, \ldots, I_{1,i;n}  ) \to
\big( \{J_{1,1}\}, \ldots, \{J_{1,i}\}  \big)  
$$
and $J_{1,1},\ldots,J_{1,i}$ are distinct points. Hence, the sets 
$I_{1,1;n},\ldots, I_{1,i;n}$ are disjoint
for all sufficiently large $n$, so it is enough co consider the case
$i=1$.

Once again, since $I_{1,1;n}/n\to \{J_{1,1}\}$ $a.s.$,  for  
almost every $\omega \in  \{ J_{1,1}\notin B\}$ we have  $I_{1,1;n}\cap
nB=\emptyset$ for all $n$ large enough. Hence for such $\omega$ both sides
of \eqref{eq;top0.0} are equal (to $-\infty$), and so we only need to show that 
\begin{equation} \label{eq;top1.2}
\frac{ \mathcal{M}_{1,1;n} (B) 
     -b_n}
   {a_n} \stackrel{P}{\rightarrow} \Lambda_{1,1} 
   \ \  \text{on } \    \{ J_{1,1} \in B\}.
\end{equation}

To this end, observe that on the event  $\{  I_{1,1;n} \cap nB \neq \emptyset \}$, 
\begin{align*}
    \mathcal{M}_{1,1;n} (B)  = & V \left(  w_n/\Gamma_1  \right) + V
                                 \left(  w_n/\Gamma_{j_{1,1;n}}    \right)  \\
    & + \max_{ t \in I_{1,1;n} \cap nB }
    \Bigl\{   
     X_t ^{(2)} + \sum_{j > j_{1,1;n} }  V \left(  w_n/\Gamma_j  \right) 1_{\{ t \in I_{j;n} \}}
    \Bigr\} . 
\end{align*}

First, it follows from \eqref{e:pi.var} that as $n\to\infty$, 
$$
  \frac{V \left(  w_n \big/ \Gamma_1   \right)    - V(w_n)}{a_n}  \to
  -\log \Gamma_1.
$$

Second, by the strong law of large numbers,
\begin{align*}
  V \left( w_n \big/ \Gamma_{j_{1,1;n}}\right) -V\left( w_n/j_{1,1;n}\right)
& = o\left( h\circ V\bigl(  w_n/j_{1,1;n}\bigr)\right) \leq o\bigl( h\circ V(w_n)\bigr).
\end{align*}
By Theorem \ref{thm;4joint} 
\begin{align*}
V\left( w_n/j_{1,1;n}\right)-V\left(  w_n
  \overline{p}_{1;n}\big/\Gamma_{1,1}\right) & =  o\left( h\circ V\left(  w_n
  \overline{p}_{1;n}\big/\Gamma_{1,1}\right) \right)  \\
  & \leq o\bigl( 
h\circ V(w_n)\bigr),  \\
V\left(  w_n
  \overline{p}_{1;n}\big/\Gamma_{1,1}\right) - 
V\left( c_\infty \vartheta_n
  Z_1^{\ast \leftarrow} (1)  \big/\Gamma_{1,1}\right) & =
  o\left( h\circ V\left( c_\infty \vartheta_n
  Z_1^{\ast \leftarrow} (1)  \big/\Gamma_{1,1}\right) \right)  \\
  & \leq  o\bigl(
h\circ V(w_n)\bigr).
\end{align*}
Apply \eqref{e:pi.var} again to get 
$$
\frac{
V\left( c_\infty \vartheta_n
  Z_1^{\ast \leftarrow} (1)  \big/\Gamma_{1,1}\right) - V ( c_\infty \vartheta_n ) }
  { h \circ V (c_\infty \vartheta_n )  }
= \log Z_1^{\ast \leftarrow} (1) - \log \Gamma_{1,1} + o (1).
$$
Due to \eqref{eq;asyV&G0.1}, we have
$$
\frac { h \circ V (c_\infty \vartheta_n )} {h\circ V(w_n)} \sim
\left( \frac{\log w_n}{\log \vartheta_n }\right)^{1/\alpha-1}
\to \frac{ 1 }{C_{\alpha,\beta}}. 
$$
This implies that 
\begin{equation} \label{e:G.11n}
\frac{V \left( w_n \big/ \Gamma_{j_{1,1;n}}\right) -V ( c_\infty
  \vartheta_n)}{a_n} \to C_{\alpha,\beta}\bigl(-\log \Gamma_{1,1} + \log Z_1^{\ast
  \leftarrow}(1)\bigr),
\end{equation}
and so, 
\begin{align*}
   & \frac{V \left(  w_n \big/ \Gamma_1   \right) + V \left( w_n \big/ \Gamma_{j_{1,1;n}}
    \right)  - b_n}{a_n}  \\
    \to &  -\log \Gamma_1 +
  C_{\alpha,\beta} ^ {-1} \bigl(-\log \Gamma_{1,1} + \log Z_1^{\ast   \leftarrow}(1)\bigr) = 
    \Lambda_{1,1}   
\end{align*}
by \eqref{eq;limRSM[0,1]}. 

Finally, we notice that the cardinality $
\# I_{1,1;n} $ is tight by 
Proposition \ref{prop;1,1card} $(\rom1)$. Because $a_n$ grows to
infinity, we see that 
$$
  \frac{
  \max_{ t \in I_{1,1;n} \cap nB  }
    \left\{   
     X_t ^{(2)} + \sum_{j > j_{1,1;n} }  V \left(  w_n/\Gamma_j  \right) 1_{\{ t \in I_{j;n} \}}
    \right\}
  }
  {a_n}  \stackrel{P}{\longrightarrow} 0. 
$$
This proves  \eqref{eq;top0.0}.
\end{proof}

\begin{proof}[Proof of Proposition \ref{prop;middle}] 
It is convenient to assume that the random sets $\{I_{1;n},n\in \bbn
\}$  and $\overline{R^\ast_1}$ are defined a probability space
$(\Omega,\mathscr{F},\PP)$, while the remaining random elements are 
defined on another probability
space, $(\Omega_1,\mathscr{F}_1,\PP_1)$, so that the overall
probability space is the product space. Clearly, $\Omega_B$ can be
viewed as an element of $ \mathscr{F}$.

Observe that, for  $\omega\in \Omega_{B}$ the random variables
$(J_{1,i})$ are $i.i.d.$ on  $(\Omega_1,\mathscr{F}_1,\PP_1)$ and each
has a positive $\PP_1$-probability to be in $B$. By Theorem
\ref{thm;4joint}, the $\PP_1$-probability that each $I_{1,i;n} $
intersects $nB$ is bounded away from zero for all large $n$. Therefore,
we can choose $K_1$ so large that with 
$$
A_{B;n} ^{(1)}  = \left\{ 
    I_{1,i;n} \cap nB \neq \emptyset
    \text{ for some }
     1\leqslant i \leqslant K_1 
    \right \}  
$$
we have 
$$
\liminf_{n\to \infty} \, \PP
\big \{ A_{B;n} ^{(1)} \,  \big|  \,  \Omega_{B} \big \}
\geqslant 1 - \epsilon/2.
$$
It is clear that on $ A_{B;n} ^{(1)}$, 
\begin{equation} \label{eq;1n-mid2.2}
    \mathcal{M}_{1;n}(B) \geqslant  V \left( w_n/\Gamma_1 \right) + 
    V \left( w_n/\Gamma_{j_{1,K_1;n}} \right) + O_P(1). 
  \end{equation}
Using once again \eqref{e:pi.var} and arguing as in the proof of
Proposition \ref{prop;top}, if we choose $c_1 >0$ sufficiently large, 
then we can make the probability 
$$
 \PP
\left \{  V \left( w_n/\Gamma_{j_{1,K_1;n}}  \right) \geqslant V(\vartheta_n) - c_1 h \circ V(\vartheta_n)  \right \}
$$
arbitrarily close to 1 as $n \to \infty$. 
Hence for some large $c_1 >0$, there is 
a sequence of subsets of $A_{B;n}^{(2)}\subseteq A_{B;n}^{(1)}$ 
that satisfy  
\begin{align}
& \liminf_{n\to \infty} \PP\{    A^{(2)}_{B;n} \, \vert \,
                \Omega_B   \} \geqslant 1-
                \epsilon,   \label{eq;1n-mid2.3} \\ 
& \liminf_{n\to \infty}
\inf_{\omega \in  A^{(2)}_{B;n} }  
\PP_\omega 
\Big \{  V \left( w_n/\Gamma_{j_{1,K_1;n}}  \right) \geqslant V(\vartheta_n) - c_1 h \circ V(\vartheta_n)  \Big \}
\geqslant 1 - \epsilon, 
\label{eq;1n-mid2.4}
\end{align} 
with $\PP_\omega(\cdot)=\ \PP\{\cdot \, \vert \,
I_{1;n},n\in \bbn ,\overline{R^\ast_1}\}$. 

\medskip
 
Next, for any $\ell$, 
\begin{align*}
    \bigvee_{ i=2^{\ell} } ^{ 2^{\ell+1}-1 }
  & \mathcal{M}_{1,i;n}([0,1]) \leqslant 
    V\left( w_n/\Gamma_1 \right)
    +
    V\left(w_n/\Gamma_{j_{1,2^{\ell};n}}  \right)  \\ 
    & + 
\max \Big \{   
     X_t ^{(2)} + \sum_{j > j_{1,2^{\ell};n} }  V \left(  w_n/\Gamma_j  \right) 1_{\{ t \in I_{j;n} \}} 
     :
      t \in  \cup_{ i=2^{\ell} } ^{ 2^{\ell+1}-1 } \widehat{I}_{1,i;n} 
    \Big \}     
    \\
    & \leqslant_{\rm st}    V \left( w_n/\Gamma_1 \right)
    +
    V \left(w_n/\Gamma_{j_{1,2^{\ell};n}}  \right)  
    + 
    \max \Big \{   
     X^{(0)}_t     :\, 1 \leqslant  t \leqslant  \sum_{ i=2^{\ell} } ^{ 2^{\ell+1}-1 } \# I_{1,i;n}    
    \Big\}, 
\end{align*}
where $\{X_t^{(0)}\}_{t\in \bbz}$ are $i.i.d.$ with $X_0^{(0)} \eid X_0$, which are also
independent of the rest two random variables on the right
hand side. It suffices to derive suitable upper bounds for the last two terms.

Take $A_{B;n} ^{(3)} =\Omega_n$ as defined in  Proposition
\ref{prop;1,1card} $(\rom2)$. We then note that 
\begin{align}
  & \liminf_{n\to \infty} \PP \left\{ A^{(3)}_{B;n} \, \Big| \,
    \Omega_{B} \right \} \geqslant 1-\epsilon \label{eq;1n-mid3.0}  
\end{align}  
and we can choose $c_2=c_2(\epsilon)$ so that 
\begin{equation}
\label{eq;1n-mid3.2}
    \lim_{\ell_0\to \infty}
    \liminf_{n\to \infty}
    \inf_{\omega \in 
    A_{B ;n} ^{(3)} }
    \PP_\omega 
    \left\{
    \frac{ \sum_{i= 2^{\ell} } ^{ 2^{\ell+1}-1 } \# I_{1,i;n} (\omega) }{2^{\ell}}
    \leqslant c_2, \,
     \ell_0 \leqslant \ell \leqslant \ell_n
    \right\} = 1. 
\end{equation}
Write $v_\ell = V(c_2 2^\ell)$.  From the facts 
that $\PP\{ X_0 > x \} \sim \overline{\nu}(x)$  and 
that $V$ is the inverse of $1/\overline{\nu}$, we have for
any positive constant $c_3$,
\begin{align*}
    & \PP \left\{
    \max_{1\leqslant t \leqslant c_2 2^ \ell} X_t^{(0)} 
    >  v_\ell + c_3 \ell h ( v_\ell )
    \right\} 
    \leqslant  c_2 2^\ell \cdot \PP 
    \left\{ 
    X_0 \geqslant 
    v_\ell + c_3 \ell h (v_\ell)
    \right \} \\
    \lesssim & \frac
    {
    \overline{H} \left(
    v_\ell + c_3 \ell h (v_\ell)
    \right)
    }
    {
     \overline{H} \left(
    v_\ell
    \right)
    }  
    =  \exp 
    \left\{
    - \int_0^1 \frac{c_3 \ell h( v_\ell ) \cdot du }{h\left( v_\ell +  c_3 \ell h  ( v_\ell) \cdot u   \right) } 
    \right\}, \quad \text{as } \ell \to \infty.
\end{align*}
By (\ref{eq;LevyMea0.1}), $h$ is eventually increasing, and we use
\eqref{eq;asyV&G0.0}  and \eqref{eq;asyV&G0.1} to verify that for large
$\ell$ the integral in the exponent is at least 
$$
    \frac{ h( v_\ell )  c_3 \ell }{h\left( v_\ell +  c_3 \ell h ( v_\ell)   \right) }
    \sim c_3\left(\frac{  \alpha \log 2}{\alpha \log 2 +
        c_3}\right)^{1-\alpha}\ell. 
$$
It follows that  
\begin{equation} \label{eq;1n-mid3.4}
\lim_{\ell_0 \to \infty}
\limsup_{n \to \infty}
\sum_{\ell=\ell_0} ^{\ell_n}
    \PP \left\{
    \max_{1\leqslant t \leqslant c_2 2^ \ell} X_t^{(0)} 
    > v_\ell + c_3 \ell h (v_\ell)
    \right\} = 0 .
  \end{equation}

Recalling the Skorohod embedding of the convergence in Theorem
\ref{thm;4joint}, we see that we can choose an event $A^{(4)}_{B} $
and some $c_4=c_4(\epsilon)>0$ such that
\begin{align}
 &  \PP\left\{  A^{(4)}_{B} \, \Big| \, \Omega_{B} \right \} \geqslant 1 -\epsilon, 
 \  \ \sup_{n\geqslant 1} \,  \sup_{ \omega \in A^{(4)}_{B}}
 \frac{w_n}{\vartheta_n} \overline{p}_{1;n}(\omega) \leqslant c_4 .
\label{eq;1n-mid4.1}
\end{align}
The upper bound on $ \overline{p}_{1;n}$ in \eqref{eq;1n-mid4.1}
guarantees the uniform convergence 
$$
\sup_{\omega \in  A^{(4)}_{B}}
\sup_{\lambda  \leq 1/2}
 \left| \EE_\omega e^{\lambda j_{1,1;n} \overline{p}_{1;n}} - \EE e^{\lambda \Gamma_{1,1}}  \right| \to 0,
$$
as in the argument for \eqref{e:Gammas}. 
Since under $\PP_\omega$ the product $j_{1,2^\ell;n} \cdot \overline{p}_{1;n}-1$
is the sum of  $2^\ell$ independent copies of  
$j_{1,1;n} \cdot \overline{p}_{1;n}-1$, the exponential Markov
inequality tells us that 
\begin{equation}\label{eq;1n-mid4.3}
\lim_{\ell_0 \to \infty}
\limsup_{n\to \infty}
\sup_{\omega \in  A^{(4)}_{B}}
\sum_{\ell=\ell_0} ^{\ell_n}
\PP_\omega \left\{
j_{1,2^\ell;n} \cdot \overline{p}_{1;n} \leqslant 2^{\ell -1}
\right\}
= 0.
\end{equation}
Combining (\ref{eq;1n-mid4.1}) with (\ref{eq;1n-mid4.3}) gives us 
\begin{equation} \label{eq;1n-mid4.4}
\lim_{\ell_0 \to \infty}
\limsup_{n\to \infty}
\sup_{\omega \in  A^{(4)}_{B;n}}
\sum_{\ell=\ell_0} ^{\ell_n}
\PP_\omega \left\{
w_n/j_{1,2^\ell;n} \geqslant 
c_4 \vartheta_n/2^{\ell -1}
\right\}
= 0
\end{equation}
and, therefore,  
\begin{equation}\label{eq;1n-mid4.5}
\lim_{\ell_0 \to \infty}
\limsup_{n\to \infty}
\inf_{\omega \in  A^{(4)}_{B}}
\PP_\omega \left\{
 V \left( w_n/\Gamma_{1,2^\ell;n}  \right) \leqslant  V \left( c_4 \vartheta_n/2^\ell  \right)
, \,
\ell_0 \leqslant \ell \leqslant \ell_n
\right\}
=  1 .
\end{equation}

Set 
$$
 A_{B;n} =  A^{(2)}_{B;n} \cap A^{(3)}_{B;n}  \cap A^{(4)}_{B}, 
 \quad n\in \bbn. 
 $$
so that 
\begin{align}
    & \liminf_{n \to \infty} \PP \left\{
    A_{B;n } \, \big| \, \Omega_{B}
  \right \} \geqslant  1- 3\epsilon, \label{eq;1n-mid1.0}.
\end{align}

Using the constants defined above, we set for  $\ell= \ell_0
,\ldots,\ell_n$ 
\begin{align*}
& T_{1,\ell}: = V(\vartheta_n) -  V \left(  c_4 \vartheta_n 2^{-\ell}
                 \right) ,  \\
& T_{2,\ell}:=c_1 h \circ V(\vartheta_n) + v_\ell + c_3 \ell h (
                                 v_\ell). 
\end{align*} 
It is elementary to check that 
\begin{align}
    &  \lim_{\ell_0 \to \infty}
    \limsup_{n \to \infty}
    \inf_{\omega \in A_{B;n} }
    \PP_{\omega}
    \left\{ \mathcal{M}_{1,n}(B) > \bigvee_{i=2^{\ell_0}} ^{2^{\ell_n} - 1} 
    \mathcal{M}_{1,i;n} ([0,1])
    \right \}  
      \label{eq;1n-mid1.1} \\ 
  \geq& - \epsilon +  \lim_{\ell_0 \to \infty}\limsup_{n \to \infty} \one \Bigl( T_{1,\ell} >
        T_{2,\ell} \ \ \text{for all $\ell= \ell_0
,\ldots,\ell_n$}\Bigr). \notag 
\end{align}
However, by \eqref{eq;asyV&G0.2}, \eqref{eq;Gprime},
\eqref{eq;asyV&G0.1}, once we take 
$0<\rho<\beta$, we see that, uniformly in $\ell$,  
\begin{align*}
  T_{1,\ell} &= 
    G(\vartheta_n )
    - G \left(
    c_4 \vartheta_n 2^{-\ell} \right) - o\left( h\circ G(\vartheta_n ) \right)    \\
    & =   \int ^ 1  _{ c_4 2^{- \ell } }
    \frac{h\circ G(\vartheta_n u)}{u} du    - o\left( h\circ G(\vartheta_n ) \right) 
              \gtrsim \ell (\log n)^{1/\alpha -1} \mathscr{L}(\log n).
\end{align*}               
On the other hand, by \eqref{eq;asyV&G0.2} and \eqref{eq;asyV&G0.1},
uniformly in $\ell$, 
\begin{align*}
 T_{2,\ell}  & \lesssim \left( \ell^{1/\alpha} + (\log n)^{1/\alpha -1}  \right) \mathscr{L}(\log n) .
\end{align*}
As long as $\rho$ is small enough, we see that the indicator function
in the right hand side of \eqref{eq;1n-mid1.1} is equal to 1 for all
$n$ large enough, and so
\begin{equation} \label{e:lowr.bond.Pw}
   \lim_{\ell_0 \to \infty}
    \limsup_{n \to \infty}
    \inf_{\omega \in A_{B;n} }
    \PP_{\omega}
    \left\{ \mathcal{M}_{1,n}(B) > \bigvee_{i=2^{\ell_0}} ^{2^{\ell_n} - 1} 
    \mathcal{M}_{1,i;n} ([0,1])
    \right \}  \geqslant 1-\epsilon.
\end{equation} 

It follows from \eqref{e:lowr.bond.Pw} and \eqref{eq;1n-mid1.0}  that
$$
 \lim_{\ell_0 \to \infty}
     \liminf_{n \to \infty}
     \PP \left\{
     \mathcal{M}_{1;n} (B)  > \bigvee_{i=2^{\ell_0} }^{2^{\ell_n} - 1} 
     \mathcal{M}_{1,i;n} ([0,1])
     \, \bigg| \,
     \Omega_B
   \right\} \geq 1-4\epsilon.
   $$
   Letting $\epsilon\to 0$ establishes \eqref{eq;1n-mid0.0}.

   \medskip

We proceed now to prove \eqref{eq;n-mid0.0}. Since $\lim_{n\to \infty}
\PP \{ I_{1;n} \cap nB \neq \emptyset \} = \PP \{ \overline{R^\ast_1}
\cap B \neq \emptyset \} >0 $, it follows that 
 $$
    \lim_{K\to \infty} \lim_{n\to \infty} 
    \PP \left\{  I_{k;n} \cap n B \neq \emptyset 
    \text{ for some }  1\leqslant k \leqslant K \right \} = 1,
  $$
Therefore, repeating the argument used to find a lower bound on
$\mathcal{M}_{1;n}(B)$ in the proof of \eqref{eq;1n-mid0.0} shows that
for any $\epsilon>0$ we have, outside of an event of probability
$\epsilon$, 
 \begin{equation} \label{eq;n-mid1.1}
    \mathcal{M}_n(B) 
    \geqslant  V (w_n/\Gamma_1) + V (\vartheta_n) - O_P (h\circ
    V(\vartheta_n)). 
  \end{equation}

Next, continuing to use the notation of  the proof of \eqref{eq;1n-mid0.0}, 
$$
\mathcal{M}_{k;n} ([0,1])
    \leqslant_{\rm st}
    V\left( w_n/\Gamma_k\right)
     + 
     \max\left\{ X^{(0)}_t : t \in I_{k;n}   \right \}, \quad k \in
     \bbn, 
$$
with the process $\{X^{(0)}_t \}$ depending on $k$, even though our
notation does not show it. For a large constant $c_5>0$, let  $v_{\ell;n} =
V\left(c_5 2^\ell \big/ \overline{F} (n)\right)$, where $F$
is the law of the first hitting time $\varphi$ in
\eqref{eq;MCmemo0.0}. For a small positive constant $c_6$, we thus have  
\begin{align*}
    & \PP \bigg \{  \bigcup_{\ell=\ell_0} ^ {\ell_n} \bigg \{ \bigvee_{k=2^{\ell}} ^ {2^{\ell+1}-1}
    \mathcal{M}_{k;n} ([0,1])  \geqslant  
    V \left(  w_n/2^{\ell - 1}  \right) +v_{\ell;n} + c_6 \ell h  (v_{\ell;n})
    \bigg \}  \bigg \} \\
    \leqslant & 
    \sum_{\ell = \ell_0} ^ {\ell_n }
    \PP \left\{  \Gamma_{2^\ell} < 2^{\ell - 1} \right \} 
    + 
    \sum_{\ell = \ell_0} ^ {\ell_n }
    \PP \left\{  
       2^{-\ell} \sum_{k=2^\ell}^{2^{\ell+1}-1}
    \# I_{k;n}
     > c_5 /\overline{F}(n)     \right \} \\
    & + 
    \sum_{\ell = \ell_0} ^ {\ell_n }
  \PP \Big\{  
   \max_{1\leqslant  t  \leqslant  c_5 2^\ell  \big/ \overline{F}(n)}  X^{(0)}_t
     \geqslant v_{\ell;n} + c_6 \ell h(v_{\ell;n})
    \Big \}   
    =: S_{1.n}+S_{2,n}+S_{3,n}. 
\end{align*}
We emphasize that the process $\{  X^{(0)}_t\}$ in $S_{3,n}$ is a
concatenation of 3 different processes with the same marginal
distribution. Only the marginal distribution is relevant in the
subsequent calculation.

An exponential Markov inequality immediately shows that 
\begin{equation} \label{eq;n-mid2.0}
    \lim_{\ell_0 \to \infty} 
    \lim_{n\to \infty}  S_{1.n}      = 0 .
  \end{equation}

  Next,  in the notation of \eqref{e:mu.p},
  choosing $c_5\geq 2\mu_1$ we have by the Chebyshev
inequality 
\begin{align*}
&\PP \left\{  
       2^{-\ell} \sum_{k=2^\ell}^{2^{\ell+1}-1}
    \# I_{k;n}
  > c_5  /\overline{F}(n)     \right \} \\
\leq &\PP \left\{  
       2^{-\ell} \sum_{k=2^\ell}^{2^{\ell+1}-1}
   \Bigl( \# I_{k;n}\overline{F}(n)  - \EE\bigl( \#
       I_{1;n}\overline{F}(n)  \bigr) \Bigr)  > c_5/2\right \}
  \\
 \leq & c2^{-\ell}{\rm Var} \bigl( \# I_{1;n}\overline{F}(n) 
        \bigr)
      \leq    c2^{-\ell}\mu_2,
\end{align*}
which implies that 
\begin{equation} \label{eq;n-mid3.1}
    \lim_{\ell_0 \to \infty} 
    \limsup_{n\to \infty}
    S_{2,n}  = 0 .
  \end{equation}

  Finally, the statement 
\begin{equation} \label{eq;n-mid4.1}
    \lim_{\ell_0 \to \infty} 
    \limsup_{n\to \infty}
    S_{3,n}  = 0 .
  \end{equation}
follows the same way as in the proof of \eqref{eq;1n-mid3.4}. By  (\ref{eq;n-mid2.0}), (\ref{eq;n-mid3.1}) and
(\ref{eq;n-mid4.1}), we note that
$$
\lim_{\ell_0 \to \infty}  \limsup_{n\to \infty}
    \PP \bigg \{  \bigcap_{\ell=\ell_0} ^ {\ell_n} \bigg \{ \bigvee_{k=2^{\ell}} ^ {2^{\ell+1}-1}
   \mathcal{M}_{k;n} ([0,1])  \leqslant  
    V \left(  w_n/2^{\ell - 1}  \right) +v_{\ell;n} + c_6 \ell h  (v_{\ell;n})
    \bigg \}  \bigg \} = 1.
    $$
    
For $\ell_0 \leqslant \ell \leqslant \ell_n$ let 
\begin{align*}
&T_{1,\ell} := V(w_n/\Gamma_1) - V (w_n 2^{-\ell+1}),    \\
&T _{2, \ell} := v_{\ell;n} + 2c_7\ell h(v_{\ell;n}) - V(\vartheta_n),    
\end{align*} 
so that for any $\epsilon>0$ 
\begin{align*}
   \PP \left\{
     \mathcal{M}_{n} (B)  > \bigvee_{k=2^{\ell_0} }^{2^{\ell_n} - 1}
  \mathcal{M}_{k;n} ([0,1])  \right\} \geqslant
\one \bigl (T_{1,\ell}>T _{2, \ell} \ \text{for all $\ell=\ell_0,\ldots,
  \ell_n$}\bigr) -\epsilon
\end{align*}
for all large $n$. 
As in the proof of \eqref{eq;1n-mid0.0}, for any
  $\epsilon>0$, on an event of probability converging to 1, 
the following two inequalities holds uniformly in $\ell$.
\begin{align*}
T_{1,\ell} 
\geqslant &   
\frac{\log 2 (1-\beta -\rho \log 2) ^ {1/\alpha - 1}}{\alpha}
\ell(\log n) ^ {1/\alpha - 1} 
\mathscr{L}(\log n)  , \\
T _{2, \ell}  \leqslant &  
 \frac{( 1 + \epsilon) (\log 2 + c_7)\big( (\beta + \rho \log 2) \big)
                          ^ {1/\alpha - 1}}{\alpha}  \ell(\log n) ^
                          {1/\alpha - 1} \mathscr{L}(\log n). 
\end{align*} 
If $c_7$, $\rho$ and $\epsilon$ are chosen to be small enough,  we
have $T_{1,\ell}>T _{2, \ell} $ uniformly in $\ell$ for large $n$. Thus 
$$
\lim_{\ell_0 \to \infty}
     \liminf_{n \to \infty}
     \PP \left\{
     \mathcal{M}_{n} (B)  > \bigvee_{k=2^{\ell_0} }^{2^{\ell_n} - 1}
     \mathcal{M}_{k;n} ([0,1])
      \right\} \geq 1-\epsilon,
     $$
     and \eqref{eq;n-mid0.0} follows by letting $\epsilon\to 0$. 
\end{proof}

\begin{proof}[Proof of Proposition \ref{prop;bottom}]  
  Due to the lower bound (\ref{eq;n-mid1.1}),  the claim
  \eqref{eq;n-btm0.0} will follow once we show that 
\begin{equation}  \label{eq;n-btm1.0}
    \lim_{n\to \infty} \PP \bigg\{ 
    V(w_n/\Gamma_1) + V(\vartheta_n) - o(V(w_n)) >
    \bigvee_{\Gamma_k > n^\rho} \mathcal{M}_{k;n} ([0,1])
    \bigg\} = 1.
\end{equation}

It suffices to consider the case $\rho<r_1$ as defined in
Lemma \ref{lem;psi} $(\rom1)$. Removing an event of probability
$\epsilon$ 
ensures that $\Gamma_1$ is bounded from above by a constant.
Modifying, if
necessary, $o(V(w_n))$ shows that \eqref{eq;n-btm1.0} will follow once
we check that
$$
\lim_{n\to \infty} \PP \bigg\{ 
    V(w_n) + V(\vartheta_n) - o(V(w_n)) >
    \bigvee_{\Gamma_k > n^\rho} \mathcal{M}_{k;n} ([0,1])
    \bigg\} = 1.
    $$

Furthermore, since  the support of the marginal L\'evy measure of the
process  
$X_0 ^{(2)}$  is bounded on the right, the process has marginal
distributional tails lighter than exponentially light;   see Section
26 in \cite{sato:2013:cam68}. It follows that
$$
  \max_{0\leqslant t \leqslant n} X^{(2)}_t  =o_P( \log n)  = o_P (V(w_n))   
  $$
  as $n\to\infty$. Therefore, it is enough to prove that
\begin{align} \label{e:another.expr}
\lim_{n\to \infty} \PP \Bigg\{ \bigvee_{\Gamma_k>n^\rho}&  \max_{t \in
     I_{k;n}}  \sum_{ j\geq k} 
 V\left( w_n/\Gamma_j \right) 1_{\{  t \in I_{j;n} \}}
  <V(w_n) + V(\vartheta_n) - o(V(w_n))   \Biggr\} =1. 
\end{align}
We prove \eqref{e:another.expr} through a series of steps.

Let $\tilde \psi$ be the function defined in Lemma \ref{lem;psi}. 
We start by proving that for any $0<r<1-\beta$ and 
any $\epsilon>0$, 
\begin{align}   \label{e:step.1}
   \PP \Bigg\{  &  
  \sum_{\Gamma_j > n^r} 
 V\left( w_n/\Gamma_j \right) 1_{\{  0 \in I_{j;n} \}}
   + V\left( w_n/n^r \right) \\
  >&V(w_n) + V(\vartheta_n) - o(V(w_n))   \Biggr\}
     \leqslant \exp\left\{ (-\tilde\psi(r)+\epsilon)\log n \right\}
 \notag 
 \end{align}
for all large $n$. 

For this purpose denote 
\begin{align*}
  &z_n = z_n (r) = V(w_n) + V(\vartheta_n)  - V(w_n / n^r) - o (V(w_n)  ) ,\\
  &\overline{z}_n = \overline{z}_n (r) =  V(w_n / n^r).   
\end{align*}
Recalling the partition $0=r_0<r_2<\cdots $ of the interval 
 $(0,1-\beta)$ defined in Lemma \ref{lem;psi}, we see that $r\in
 (r_m.r_{m+1}]$ for some $m=0,1,2\ldots$, so by the same Lemma
 \ref{lem;psi},  for large $n$, the probability in \eqref{e:step.1} is
\begin{align} 
& \PP \Bigg\{  
  \sum_{\Gamma_j > n^r} 
 V\left( w_n/\Gamma_j \right) 1_{\{  0 \in I_{j;n} \}}>z_n (r)\Biggr\}  \notag \\
  \leqslant &  \overline{z}_n(r) ^ {\gamma_m} \big( \overline{H}(\overline{z}_n(r))
        \big) ^ m   \overline{H}( z_n(r) -m \overline{z}_n(r)).  \label{e:step1.1}
\end{align}
 It follows from $(\ref{eq;asyV&G0.0})$ that 
\begin{equation} \label{eq;revised;n-btm2.1.2}
\overline{z}_n(r) ^ {\gamma_m}  \lesssim  \left( (\log n) ^ {1/\alpha }
  \mathscr{L}(\log n) \right ) ^ {\gamma_m}  .
\end{equation}
Next, by Karamata's theorem, \eqref{eq;asyV&G0.0} and
\eqref{eq;asyV&G0.1} 
\begin{align} 
   \int_1 ^ {  \overline{z}_n(r) }  \frac{du}{u^{1-\alpha} L_\alpha(x) }  
    \sim &  \frac{ \overline{z}_n(r) }{ \alpha h ( \overline{z}_n(r) )   } \notag \\
     =  &  \frac{1}{\alpha}  \cdot \frac{ V(w_n / n^r ) }{ h \circ V(w_n
   / n^r )  } \sim \log(w_n/n^r) \sim (1-\beta - r) \log n. \label{eq;revised;n-btm2.2.2}
\end{align}
Let
$$
\theta= \big(   ( 1-\beta )^{1/\alpha} + \beta ^{ 1/\alpha }
- (m+1) (1-\beta -r)^{1/\alpha}   \big)^\alpha; 
$$
due to the range of $r$ this is a well defined power of a positive
number. It follows by $(\ref{eq;asyV&G0.0})$ that 
\begin{align*}
    & z_n - m \overline{z}_n     \sim  V( n^\theta ), 
\end{align*}
Therefore, by \eqref{eq;asyV&G0.0} and \eqref{eq;asyV&G0.1}, 
\begin{align} \label{eq;revised;n-btm2.3.0}
  &  \int_1 ^ { z_n - m \overline{z}_n  }  \frac{du}{u^{1-\alpha}
    L_\alpha (u)}  
  \sim  \frac{ z_n - m \overline{z}_n }{ \alpha h  (z_n - m
    \overline{z}_n) } 
  \sim \frac{V(n^\theta)}{\alpha h\circ V(n^\theta)} 
     \sim  \theta \log n .
\end{align}
Putting \eqref{e:step1.1}, \eqref{eq;revised;n-btm2.1.2},
\eqref{eq;revised;n-btm2.2.2} and \eqref{eq;revised;n-btm2.3.0}
together establishes \eqref{e:step.1}. 

Next, we prove that \eqref{e:another.expr} holds if $\rho\in
(r_m,r_{m+1}]$ for any sufficiently large $m$. Indeed, by
\eqref{e:step.1}, 
 \begin{align*}
   &   \PP \Bigg\{ \bigvee_{\Gamma_k>n^\rho}  \max_{t \in
     I_{k;n}}  \sum_{j\geq k} 
 V\left( w_n/\Gamma_j \right) 1_{\{  t \in I_{j;n} \}}
  \geqslant V(w_n) + V(\vartheta_n) - o(V(w_n))   \Biggr\} \\
 \leqslant & n \PP \Bigg \{ \sum_{\Gamma_j > n^r} V\left(
             \frac{w_n}{\Gamma_j} \right) 1_{ \{ 0 \in I_{j;n} \}  } + V (w_n / n^r )
             \geqslant 
 V(w_n) + V(\vartheta_n) - o ( V(w_n))  \Biggr\}  \\
  \leqslant & n \exp\left\{     \big( - \widetilde{\psi}(r) + \epsilon  \big) \log n   \right \}
 \end{align*}
for large $n$. Since $\widetilde{\psi}(r)\to\infty$ as $r\to 1-\beta$,
the claim follows. 

Given that \eqref{e:another.expr} holds if $\rho\in
(r_m,r_{m+1}]$ for any sufficiently large $m$, the fact that 
\eqref{e:another.expr} also holds for any $0<\rho<r_1$ will follow
once we prove that for any such $\rho$ and any $m$, 
\begin{align*} 
\lim_{n\to \infty} \PP \Bigg\{ \bigvee_{n^\rho<\Gamma_k\leqslant
  n^{r_m}}&  \max_{t \in 
     I_{k;n}}  \sum_{j\geq k} 
 V\left( w_n/\Gamma_j \right) 1_{\{  t \in I_{j;n} \}} \\
 & \hskip 1in <V(w_n) + V(\vartheta_n) - o(V(w_n))   \Biggr\} =1.  
\end{align*}

We will prove the above statement by constructing two increasing sequences, 
$$
\{ i_s \}_{ 0 \leqslant s \leqslant m } \subseteq \bbn 
\quad \text{and} \quad 
\{ \rho_ {i_s} \}_{ 0 \leqslant s \leqslant m } \subseteq \reals_+,
$$

with $i_0:=0$, $\rho_0 := \rho$, 
   $\rho_{i_s} \in (r_s, r_{s+1})$ for $s=1,\ldots, m$ such that for
   every such $s$, 
     \begin{align} \label{eq;revised;n-btm4.1.0}
  \lim_{n\to \infty} \PP \Bigg\{ \bigvee_{n^{\rho_{i_{s-1}}}<\Gamma_k\leqslant
  n^{\rho_{i_s}}}&  \max_{t \in 
     I_{k;n}}  \sum_{j\geq k} 
 V\left( w_n/\Gamma_j \right) 1_{\{  t \in I_{j;n} \}} \\
 & \hskip 1in \geqslant V(w_n) + V(\vartheta_n) - o(V(w_n))   \Biggr\}
   =0. \notag   
\end{align}
We will describe the construction in the case $m=1$. The case of a
general $m$ is similar.

Let $\delta_0  = \widetilde{\psi} (\rho_0) - (  \rho_0 + \beta ) >
0$, by  Lemma \ref{lem;psi} $(\rom2)$. So 
    $  \widetilde{\psi} (\rho_0) >  \rho_0 + \beta + 2\delta_0/3$, and so
    using  \eqref{e:step.1} and the notation that follows it,  
\begin{equation}\label{eq;revised;n-btm4.2.2}
\PP\biggl\{
\sum_{\Gamma_j > n^{\rho_0}} V \left(  w_n/\Gamma_j \right) 1_{ \{  0 \in I_{j;n} \} }
\geqslant z_n (\rho_0)
\biggr\} = o \left(  n^ { -\rho_0 - \beta - 2\delta_0 / 3 } \right). 
\end{equation}
 Take
$\rho_1  = \rho_0 +  \delta_0/3$, then for any $c>0$, 
\begin{align*}
  &\PP \Bigg\{ \bigvee_{n^{\rho_0}<\Gamma_k\leqslant
  n^{\rho_1}}  \max_{t \in 
     I_{k;n}}  \sum_{j\geq k} 
 V\left( w_n/\Gamma_j \right) 1_{\{  t \in I_{j;n} \}}  
  \geqslant V(w_n) + V(\vartheta_n) - o(V(w_n))   \Biggr\} \\
  \leqslant & \PP \big \{ \#\{ k:  n^{\rho_0} < \Gamma_k \leqslant
              n^{\rho_1}  \} > 2 n^{\rho_1} \big \}   \\
+ & \PP \Big \{ \# I_{k;n} \geqslant \frac{c n^\beta \log n}{L(n)} \ 
    \text{ for some $k$ with} \ n^{\rho_0} < \Gamma_k \leqslant n^{\rho_1}  \Big \} \\
 + &2n^{\rho_1}  \frac{cn^\beta \log n}{L(n)} \cdot \PP \bigg \{
     \sum_{\Gamma_j > n^ {\rho_0 } }   V\left(  w_n/\Gamma_j \right) 1_{\{  0 \in I_{j;n} \}}
    > z_n (\rho_0)
    \bigg \}. 
\end{align*}
The first term in the right hand side vanishes in the limit by the law
of large numbers. If $c$ is large enough, by  Proposition
\ref{prop;ALLcard} so does the second term.  The same is true for the third
  term by 
  $(\ref{eq;revised;n-btm4.2.2})$. Notice that, if $\rho_1 > r_1$, we set $i_1=1$ and replace
  $\rho_1$ with $\min(\rho_1,r_2)$ to 
  complete the construction of the two sequences. 
Otherwise,  we inductively define, for as long as $\rho_i\leqslant
r_1$, 
\begin{equation} \label{eq;revised;n-btm4.3.2}
 \delta_{i}=  \widetilde{\psi} (\rho_i) - (\rho_i+\beta),   
 \quad 
 \rho_{i+1} = \rho_i + \frac{i+1}{i+3} \delta_{i+1}, \quad i \geqslant 1.
\end{equation}
The same argument as above then  shows that for every $i$, 
$$
\PP \Bigg\{ \bigvee_{n^{\rho_i}<\Gamma_k\leqslant
  n^{\rho_{i+1}}}  \max_{t \in 
     I_{k;n}}  \sum_{j\geq k} 
 V\left( w_n/\Gamma_j \right) 1_{\{  t \in I_{j;n} \}}  
 \geqslant V(w_n) + V(\vartheta_n) - o(V(w_n))   \Biggr\}\to 0.
 $$
We claim that this inductive procedure must end after finitely many
steps. That is, $\rho_{i_1} > r_1$ for some $i_1 \in \bbn$, in which case 
we replace $\rho_{i_1}$ with $\min(\rho_{i_1},r_2)$
and, once again, complete the construction. 

Indeed, if the inductive procedure did not terminate, 
it would follow that $\delta_i \to 0$, so $\rho_i\uparrow
\rho_*\leqslant r_1$ satisfying
$\widetilde{\psi}(\rho_*-)=\rho_*+\beta$. Since $\widetilde{\psi}$ is
constant on $(0,r_1)$, this is easily seen to contradict the property
$r_1<1/2-\beta$ established in Lemma  \ref{lem;psi} $(\rom1)$. 

This completes the proof of \eqref{eq;n-btm0.0}.

We proceed to
prove \eqref{eq;1-btm0.0}. 
Let $\epsilon>0$. For $c,C>0$ we set  
$\Omega_n = \Omega_n^{(1)} \cap \Omega_n^{(2)} $,
where 
\begin{align*}
   & \Omega_n^{(1)}= \Big\{ \omega\in \Omega:  \min\bigl( \Gamma_{ j_{1,k;n}  }(\omega) ,  j_{1,k;n}(\omega)\bigr) \geqslant  c w_n k/\vartheta_n  
 \ \text{ for all }  k \in \bbn   \Big\}, \quad n\in \bbn;   \\ 
   & \Omega_n^{(2)} = \Big\{
\# I_{1,k;n} (\omega) < C \log n  \text{ for all } 1\leqslant  k \leqslant \vartheta_n / c
\Big\}, \quad n \in \bbn. 
\end{align*}
If $\delta>0$ is such that $V$ vanishes on $(0,\delta]$, then 
on $\Omega_n^{(1)}$ we have $$\#\{ k : \Gamma_{ j_{1,k;n}
} \leqslant w_n/\delta \} \leqslant \vartheta_n / (c\delta).$$
 We claim that we can choose $c=c(\epsilon)$ small enough and $C = C(\epsilon)$ large enough so that 
\begin{equation}\label{eq;1-btm2.3} 
\liminf_{n\to \infty} \PP \big\{ \Omega_n^{(i)}  \big\} \geqslant 1 - \epsilon, \quad
i=1,2.
\end{equation}
The fact that this is true for $i=1$ if $c$ is small enough follows as
in (\ref{eq;1n-mid4.4}) and the law of large numbers, for $i=2$ it
follows for large enough $C$ by 
Proposition \ref{prop;1,1card} $(\rom2)$.

With $c$  chosen above we denote 
 \begin{align*} 
   & \overline{ \mathcal{M} }_{1,k;n} = \max_{t \in I_{1,k;n}} \bigg\{
    V\left(w_n/\Gamma_{j_{1,k;n}} \right) + \sum_{j \geqslant j_{1,k;n} + 1}
    V\left(  w_n/\Gamma_j \right) 1_{\{ t \in I_{j;n} \}} \bigg\},
     \quad k\in \bbn; \\
   &    \mathcal{M}_{1,k;n}^\prime = \max_{t \in I_{1,k;n}} \bigg\{
    V\left(
     \vartheta_n/(c k)\right) + 
    \sum_{\Gamma_j >   c k w_n / \vartheta_n }
    V \left(
    w_n/\Gamma_j  \right) 1_{\{ t \in I_{j;n} \} }\bigg\}   , \quad 
     k\in \bbn, 
\end{align*}
so that on each event $\Omega_n$, we have
$\overline{\mathcal{M}}_{1,k;n} \leqslant \mathcal{M}_{1,k;n}^\prime$,
$   k \geqslant 1$. 
Since $c_\infty<1$, it follows from \eqref{e:G.11n} that on the event
$\Omega_B$, 
$$
V\left(  w_n/\Gamma_{j_{1,1;n}} \right) \geqslant 
V(\vartheta_n) - o_P(V(\vartheta_n)).
$$
Since $\epsilon>0$ can be taken as small as we wish,
(\ref{eq;1-btm0.0}) will follow from the following claim:  
\begin{equation}\label{eq;1-btm3.1}
    \lim_{n\to \infty} \mathbb{P} \bigg \{ \max_{\lfloor n^\rho
      \rfloor \leq k\leq \lceil \vartheta_n / (c\delta) \rceil}
    \mathcal{M}_{1,k;n}^\prime <
V(\vartheta_n) - o(V(\vartheta_n)) \, \bigg| \, \Omega_n
\bigg \}  = 1 .
\end{equation}

We start by constructing an increasing sequence $\rho_i \uparrow
\beta, i \in \bbn_0$ with $\rho_0=\rho$ such that for every $i$
\begin{equation} \label{eq;revised;1-btm2.3.0}
     \mathbb{P} \left\{
    \max_{\lfloor n^{\rho_i} \rfloor \leqslant k\leqslant  \lceil n ^ {\rho_{i+1}} \rceil }
    \mathcal{M} ^ \prime _{1,k;n}  
    \geqslant V(\vartheta_n ) - o(V(\vartheta_n ) )   \, \bigg| \, \Omega_n
    \right \}  \to 0 , \quad n\to \infty. 
\end{equation}
To this end,  we define inductively for $i=0,1,\ldots$ 
\begin{equation} \label{eq;revised;1-btm2.3.1}
  \delta_i =  \left(  \beta ^ { 1/\alpha }- (\beta - \rho_i)^
    { 1/\alpha } \right) ^ \alpha
    - \rho_i, \quad 
  \rho_{i+1}  = \rho_i + \frac{i+1}{i+3} \delta_i ;
\end{equation}
note that $\delta_i>0$ for all $i$. We have 
\begin{align*}
  &   \PP \left \{   V\left(
    \vartheta_n/(c n^{\rho_i})  \right) + 
    \sum_{\Gamma_j >   c n^{\rho_i} w_n / \vartheta_n }
    V \left(
    w_n/\Gamma_j  \right) 1_{\{ 0 \in I_{j;n} \} }  \geqslant  V(\vartheta_n) - o(V(\vartheta_n))  \right  \} \\
    \leqslant & 
    \PP \left \{ 
    \sum_{j=0 } ^ \infty
    V \left(
    w_n/\Gamma_j  \right) 1_{\{ 0 \in I_{j;n} \} }  \geqslant  V(\vartheta_n) - o(V(\vartheta_n)) 
    -   V\left(
    \vartheta_n/(c n^{\rho_i})
    \right) 
    \right  \} \\
    \lesssim & \overline{H}  \Bigl(
    V(\vartheta_n) - V \left(  \vartheta_n/(cn^{\rho_i}) \right) -
               o\big( V(\vartheta_n) \big)       \Bigr), 
\end{align*}
and by \eqref{eq;asyV&G0.0},
$$
\log\Bigl[ \overline{H}  \Bigl(
    V(\vartheta_n) - V \left(  \vartheta_n/(cn^{\rho_i}) \right) -
    o\big( V(\vartheta_n) \big)       \Bigr)\Bigr]
 \sim - \left(  \beta ^ { 1/\alpha }- (\beta - \rho_i)^
   { 1/\alpha } \right) ^ \alpha \log n.
 $$
 Therefore, by the first part of \eqref{eq;revised;1-btm2.3.1}, for
 large $n$ we have 
 $$
 \overline{H}  \Bigl(
    V(\vartheta_n) - V \left(  \vartheta_n/(cn^{\rho_i}) \right) -
    o\big( V(\vartheta_n) \big)       \Bigr)
    \leqslant \exp \left\{ - \left(\rho_i +  \frac{i+2}{i+3}\delta_i
      \right) \log n  \right \} 
$$
and, hence, since $\Omega_n\supseteq\Omega_n^{(2)}$, 
\begin{align*}
    &\mathbb{P} \left\{
    \max_{\lfloor n^{\rho_i} \rfloor \leqslant k\leqslant  \lceil n ^ {\rho_{i+1}} \rceil }
    \mathcal{M} ^ \prime _{1,k;n}  
    \geqslant V(\vartheta_n ) - o(V(\vartheta_n ) )   \, \bigg| \, \Omega_n
    \right \} 
      \\
    \lesssim &  n^{\rho_{i+1}} \cdot C \log n  \cdot  \exp \left\{ - \left(\rho_i +  \frac{i+2}{i+3}\delta_i \right) \log n 
    \right \} \\
    = &  C \log n \cdot  \exp\left\{ - \frac{\delta_i}{i+3} \log n \right\}
\to 0 
\end{align*}
as  $n \to \infty$.
Therefore, (\ref{eq;revised;1-btm2.3.0}) follows, and
\eqref{eq;1-btm3.1} will be established once we show that for all $i$
large enough,  
\begin{equation} \label{eq;revised;1-btm2.4.0}
\mathbb{P} \left \{
    \max_{\lfloor n^{\rho_i} \rfloor \leqslant k\leqslant \lceil \vartheta_n / c \rceil }
   \mathcal{M} ^ \prime _{1,k;n}  
    \geqslant  V(\vartheta_n ) - o(V(\vartheta_n ) )   \, \bigg| \, \Omega_n
    \right \} \to 0, \quad n \to \infty. 
\end{equation}
Note that
$$
\sum_{\Gamma_j >   c n^{\rho_i} w_n / \vartheta_n }
    V \left(
    w_n/\Gamma_j  \right) 1_{\{ 0 \in I_{j;n} \} }
\eid  \sum_{j=1}^{N_n} \overline{\xi_j},    
$$
where $N_n$ is a Poisson random variable with mean 
$\overline{\nu}(x_0)-cn^{\rho_i}/\vartheta_n$, and 
$\{ \overline{\xi}_i \}_{i\geqslant 1}$ is a family of $i.i.d.$ random
variables independent of $N_n$, and the law of $\overline{\xi}_1$  is
the restriction of $\nu$  to $\bigl(1, V( \vartheta_n c^{-1} n^{-\rho_i})
\bigr)  $, normalized to be a probability measure. Here we have used
the fact that we assume $x_0=1$ in \eqref{eq;x0}.

Let 
$$
m = m(i) = \left \lceil   \frac{  V (\vartheta_n ) - o(V (\vartheta_n
    )) -  V( \vartheta_n c^{-1} n^{-\rho_i})} { V( \vartheta_n c^{-1}
    n^{-\rho_i})  }   \right \rceil;   
$$
$m(i)$ also depends on $n$, but by \eqref{eq;asyV&G0.0} it is bounded as
a function of $n$. By Lemma \ref{lem;CvxAls} $(\rom1)$ we have for
some large numbers $\gamma_1(i),\, \gamma_2(i)$, 
\begin{align*}
   &\PP \left\{ \sum_{\Gamma_j >   c n^{\rho_i} w_n / \vartheta_n }
    V \left(
    w_n/\Gamma_j  \right) 1_{\{ 0 \in I_{j;n} \} }>  V (\vartheta_n ) - o(V (\vartheta_n
    )) -  V( \vartheta_n c^{-1} n^{-\rho_i})\right\} \\
   \leqslant & \sum_{ s = m } ^ \infty 
    \frac{ 1 }{ s !  }  \PP \left\{ \sum_{j=1}^s  \overline{\xi}_j \geqslant  V (\vartheta_n ) - o(V (\vartheta_n )) - V( \vartheta_n c^{-1} n^{-\rho_i})   \right \} \\
    \leqslant & \sum_{ s = m } ^ \infty 
    \frac{ 1 }{ s !  }   \gamma_1(i)^s 
     \left( V( \vartheta_n c^{-1} n^{-\rho_i}  ) \right) ^
                {\gamma_2(i)}  
    \left(  \overline{H} ( V( \vartheta_n c^{-1} n^{-\rho_i}  )   )   \right) ^ {m(i)} \\
    & \quad \quad \cdot \overline{H} \big(  V(\vartheta_n) - o (V(\vartheta_n)) - m V( \vartheta_n c^{-1} n^{-\rho_i}  )  \big ) 
    \\
    \lesssim & \left( V( \vartheta_n c^{-1} n^{-\rho_i}  ) \right) ^
               {\gamma_2(i)}     \left(  \overline{H} ( V( \vartheta_n
               c^{-1} n^{-\rho_i}  )   )   \right) ^ {m(i)}
    \lesssim           (\log n)^{\gamma_2(i)/\alpha} \bigl(
               \vartheta_n n^{-\rho_i}\bigr)^{-m(i)},  
\end{align*}
where we have once again used \eqref{eq;asyV&G0.0}. Furthermore, 
$$
 m(i)\to 
   \frac{\beta^{1/\alpha} - (\beta - \rho_i)^{1/\alpha} }{ (\beta - \rho_i)^{1/\alpha}}
$$
as $n\to\infty$, so that for large $n$, 
$\bigl(\vartheta_n n^{-\rho_i}\bigr)^{-m(i)} \leqslant n^{-C_1}$
for
$$
C_1= (\beta - \rho_i) \frac{\beta^{1/\alpha} - (\beta - \rho_i)^{1/\alpha} }{ (\beta - \rho_i)^{1/\alpha}} .
$$
Since $\rho_i\to\beta$ as $i\to\infty$ and $\alpha<1$, $C_1$ will
become arbitrarily large for large $i$. Taking into account the bound
on the cardinality of $I_{t,k;n}$ on the event $\Omega_n$, we see that 
for  $i$ such that $C_1>\beta$  and for large $n$, 
\begin{align*}
    &  \mathbb{P} \left \{
    \max_{\lfloor n^{\rho_i} \rfloor \leqslant k\leqslant \lceil \vartheta_n / c \rceil }
   \mathcal{M} ^ \prime _{1,k;n}  
    \geqslant  V(\vartheta_n ) - o(V(\vartheta_n ) )   \, \bigg| \, \Omega_n
      \right \}  \\
     \lesssim &  \vartheta_n  \log n \cdot  n^{-C_1}     \to  0 
\end{align*}
as $n\to\infty$. This establishes \eqref{eq;revised;1-btm2.4.0} and,
hence, completes the proof. 
\end{proof}

\begin{proof}[Proof of Theorem \ref{thm;EPcv}]
Recall that the random sup-measure $\mathcal{M}$
is coupled to the random sup-measures $(\mathcal{M}_n)$ as described
when proving Theorem \ref{thm;RSMcv}. Therefore, 
  it suffices to show that for any fixed $0<T_1<T_2 \leqslant 1$,
\begin{equation} \label{eq;EPcv1.0}
    \left\{
    \frac{ \mathbb{M}_n(t)- b_n}{a_n}
    \right\}_{t\in [T_1,T_2]} \stackrel{P}{\longrightarrow}
    \{ \mathbb{M}(t) \}_{t\in [T_1,T_2]} \quad
    \text{in } \big(D[T_1,T_2], J_1 \big).
\end{equation}
For the duration of the proof we fix a small $\epsilon>0$.

We start by observing that by the construction of $\mathbb{M}(t)$ in
(\ref{eq;limRSM0.0}), (\ref{eq;limRSM[0,1]})  
and (\ref{eq;EP0.0}), there exists $N=N(\epsilon)>0$ such that the
event  
\begin{equation}\label{eq;renew;EPcv2.1}
    \Omega^{(1)} : = \left\{   
    \mathbb{M}(t) = \max_{ 1 \leqslant k,i \leqslant N } \big \{
    \Lambda_{ k , i }:  W_{k,i} \in [0,t]\bigr\} \ \text{for $t\in
      [T_1,T_2]$} 
    \right \} 
\end{equation}
satisfies $ \PP\{ \Omega^{(1)}  \} \geqslant 1- \epsilon $. 
Consider the random finite set 
 $\mathcal{T} = \{  W_{k,i} : 1 \leqslant k, i \leqslant N   \} \cup
 \{ 0, T_1, T_2  \}$. Since the random measure $\eta_\omega$ in
 \eqref{eq;eta_omega} is a.s. atomless, it follows that there is
 $\delta>0$ such that   the event 
\begin{equation}\label{eq;renew;EPcv2.2}
    \Omega^{(2)} : =\big\{ \vert t_1 - t_2 \vert \geqslant \delta  \text{ for all distinct } t_1, t_2 \in \mathcal{T}    \big  \}   
\end{equation}
 satisfies $ \PP\{ \Omega^{(2)}  \} \geqslant 1- \epsilon $. Then on 
 $\Omega^{(1)} \cap \Omega^{(2)} $ the process 
 $\{ \mathbb{M}(t): t\in [T_1, T_2] \}$  has nondecreasing piecewise
 constant sample paths, with at most $N^2$ jumps, and each two jump
 points are separated by at least $\delta$. Let
 $T_1=W_{k_0,i_0}<W_{k_1,i_1}<\cdots<W_{k_\ell,i_\ell}\leqslant T_2$
 be the enumeration of the jumps. 
 
We saw when establishing \eqref{eq;top1.2}  with
$B=[0,1]$, that for any $k,i\in \bbn$, 
$$
    \max_{ t\in I_{k,i;n}  }  \left \vert 
    \frac{ X(t)  - b_n  }{a_n}  - \Lambda_{k,i}
    \right \vert \stackrel{P}{\to} 0 \quad \text{as } n \to \infty.
 $$
Therefore by decreasing $\delta$ if necessary, the event 
\begin{align*}
  \Omega^{(3)}_n : = \Bigg \{ & 
 \left \vert   \frac{ \mathbb{M}_n ([0,T_i]) - b_n }{a_n} -
                                \mathbb{M}(T_i)   \right \vert
                                \leqslant \delta ,  \, i=1,2,\\
     &   \bigvee_{ k,i= 1} ^{ N } 
     \max_{ t\in I_{k,i;n}  }  \left \vert 
    \frac{ X(t)  - b_n  }{a_n}  - \Lambda_{k,i}
    \right \vert   \leqslant \delta   \Bigg \}  
\end{align*}
satisfies $\PP\{ \Omega^{(3)}_n \} \geqslant 1 - \epsilon$ for all
large $n$. Furthermore, by Proposition
\ref{prop;middle} and Proposition \ref{prop;bottom},  for $N$ large
enough, the event 
\begin{equation}\label{eq;renew;EPcv3.3}
\Omega^{(4)}_n : = 
\left\{
\mathbb{M}_n (t) = \max_{1 \leqslant k, i \leqslant N}  \; \max_{ t\in I_{k,i;n} \cap [0, nt] }     X(t)   
\  \text{ for all } \  t\in [T_1, T_2].
\right \}  
\end{equation}
satisfies $\PP\{ \Omega^{(4)}_n \} \geqslant 1 - \epsilon$ for all
large $n$. It is clear that we can select the same $N$ in
\eqref{eq;renew;EPcv2.1} and \eqref{eq;renew;EPcv3.3}. Finally, since
for every $k,i$ 
$$
  I_{ k,i;n }/n  \longrightarrow \{ W_{k,i}\} \ \text{a.s.}
$$
with respect to the Hausdorff metric as $n\to\infty$, the event 
\begin{equation}\label{eq;renew;EPcv3.4}
\Omega^{(5)}_n : = 
\left\{  \max_{ 1\leqslant k,i  \leqslant N } \; 
\max_{ t\in I_{k,i;n}  }
| t/n  - W_{k,i} |
\leqslant \delta/2
\right \}  
\end{equation}
satisfies $\PP\{ \Omega^{(5)}_n \} \geqslant 1 - \epsilon$ for all
large $n$.

Fix $\omega\in \Omega^{(1)} \cap \Omega^{(2)} \cap \Omega^{(3)}_n \cap
\Omega^{(4)}_n \cap \Omega^{(5)}_n$. We construct  a function $e_n:\,
[T_1,T_2]\to [T_1,T_2]$ as follows. First set 
$$
e_n (t) = 
\begin{cases}
T_1 \quad & t = T_1 \\
W_{ k_s , i_s }  & t =  \min \bigl( I_{ k_s , i_s ; n }/n\bigr)\  \text{ for } s=1,\ldots,\ell   \\
T_2, & t = T_2
\end{cases}
$$
and note that due to the choice of $\omega$ this an increasing
function. We extend it to a continuous increasing map from $[T_1,T_2]$
onto $[T_1,T_2]$ by linear interpolation. Then 
\begin{align*}
    & \sup_{t\in [T_1, T_2] } \vert e_n(t) - \text{id}(t) \vert   
    =\bigvee_{s=1}^\ell\left \vert  
    \min\bigl(I_{k_s, i_s;n}/n\bigr)  - W_{k_s, i_s}   \right \vert
      \leqslant  \delta/2
\end{align*}
so that 
\begin{align*}
    & d_{ J_1 } \left (   
    \left\{
    \frac{ \mathbb{M}_n(t)- b_n}{a_n}
    \right\}_{t\in [T_1,T_2]} , 
    \{ \mathbb{M}(t) \}_{t\in [T_1,T_2]} 
    \right )  \\ 
    \leqslant &   \sup_{t\in [T_1, T_2] } \vert e_n(t) - \text{id}(t) \vert   \bigvee   
    \sup_{t\in [T_1, T_2]} \left \vert 
    \frac{ \mathbb{M}_n (t ) - b_n  }{a_n} 
    - \mathbb{M}(e_n(t))
    \right \vert 
    \\ 
    \leqslant & \frac{\delta}{2} \vee  \left( \bigvee_{i=1}^2\left \vert   \frac{ \mathbb{M}_n ([0,T_i]) - b_n }{a_n} - \mathbb{M}(T_i)   \right \vert \right) 
     \vee  \left(  \bigvee_{ s= 1} ^{\ell }
     \max_{ t\in I_{k_s,i_s;n}  }  \left \vert 
    \frac{ X_t - b_n  }{a_n}  - \Lambda_{k_s,i_s}
    \right \vert \right ) \\
    \leqslant &  \delta.
\end{align*}
Since $\delta>0$ is arbitrary, while 
$$
\liminf_{n\to \infty}
 \PP\left\{  \Omega^{(1)} \cap \Omega^{(2)} \Omega^{(3)}_n \cap  \Omega^{(4)}_n \cap \Omega^{(5)}_n \right \} \geqslant 1 - 5 \epsilon, 
 $$
 and $\epsilon>0$ is also arbitrary, \eqref{eq;EPcv1.0} follows. 
\end{proof}

\appendix

\section{Ranges and Local Times  }  \label{sec;AppA}

This appendix is largely devoted to proving Theorem \ref{thm;4joint}. Unless otherwise stated, $F$ denotes the distribution in Assumption \ref{ass;MCmemo}.

Let us begin by looking closely at the set $I_{0;n}$ in (\ref{eq;Ikn}). By the
definition of the probability measure $\mu_n$, the random function
$Y^{(k;n)}$ has its first zero coordinate between $0$ and $n$, and the 
subsequent zero coordinates appear whenever the Markov chain returns
back to zero. Therefore, $I_{0;n}$ is, in distribution, the
restriction to $\{0,\ldots,n\}$ of the range of a random walk with the
step distribution $F$, starting at a random position in
$\{0,\ldots,n\}$. The following definition formalizes this type of
random walks, which we will study in this appendix. 

\begin{assumption} \label{ass;RVRW}
For $n\in\bbn_0$, $\mathbf{S}^ {(n)} =\{ S^{(n)}_t\}_{t\in \bbn_0}$ is
a random walk such that 
\begin{enumerate}[label=(\alph*)]
    \item the initial position  $S^{(n)}_0$ has the law of $\min \{ 0 \leqslant t \leqslant n : Y^{(n)}_t  = 0 \}$, where 
    $\{Y^{(n)}_t\}_{t \in \bbz}$ has the law $\mu_n$ in (\ref{eq;mu_n});
    \item the steps $\{\xi_t\}_{t\in \bbn }$ are $i.i.d.$ with
      distribution $F$, and are independent of $S^{(n)}_0$. 
\end{enumerate}
\end{assumption}
That is, the different random walks in Assumption \ref{ass;RVRW} only
differ in the law of the initial position. Hereafter, $\{ \{
S^{(k;n)}_t\}_{t\in \bbn_0}   \}_{k\in \bbn_0}$ 
denotes a collection of $i.i.d.$ copies of $\mathbf{S}^ {(n)} $. 
For the duration of this section we work with the sets $I_{k;n}$ in
(\ref{eq;Ikn}) written as
\begin{equation}  \label{eq;Ikn*}
    I_{k;n} = \bigl\{ S^{(k;n)}_t : t \in \bbn_0  \} \cap \{0,\ldots,n
    \bigr\},  \ \   n,k \in \bbn_0.
  \end{equation}
By the definition these are nonempty random sets.  

We will need several facts about these random walks, and we list these facts in
the proposition below. Most of them are well known. In the sequel we
use the notation $\min A$ ($\max A$) to denote the smallest (largest)
point in a discrete set $A$. 
\begin{proposition}\label{prop;S_n-Pre}
\begin{enumerate} [label=(\roman*)]
    \item For every $k\in \bbn_0$,
    \begin{equation} \label{eq;S_n-Pre0.0}
    \left( \frac{\min \, I_{k;n}}{n} ,  \frac{I_{k;n}}{n}    \right)\Rightarrow  \left( Z^\ast (0),  \overline{R^\ast} \right) 
    \quad \text{weakly in } \   [0,1] \times  \mathcal{F}([0,1])
  \end{equation}
as $n\to\infty$, where $Z^\ast(0)$ and $\overline{R^\ast}$ are
connected by (\ref{eq;shiftSUB0.1}), \eqref{eq;shiftSRS}  and
(\ref{eq;barSRS}). 
    
    \item If $A_0$ denotes the full range $\{ S^{(0)}_0,
      S^{(0)}_1,\ldots  \}$ of $\mathbf{S}^ {(0)}$, then 
    \begin{equation} \label{eq;S_n-Pre0.1}
     \PP \left\{ n \in A_0  \right \}  \sim \frac{n^{\beta -1} 
     \overline{F}(0) }{ \Gamma(\beta) \Gamma(1-\beta) L(n) }    \quad 
   \text{as $n\to\infty$} 
     \end{equation}
     and
    \begin{equation} \label{eq;S_n-Pre0.2}
    \limsup_{n_0 \to \infty}  \, 
    \sup_{n > n_0} \,
    \max_{0 \leqslant k \leqslant n-1} 
    \max_{m \in \bbz}
    \frac{ \# \big( A_0 \cap [m, m+2^k) \cap [2^{n_0}, 2^n) \big) }{ n \big/ \overline{F}(2^k) }
    < \infty  \quad  a.s.
    \end{equation}

    \item For any $\gamma,\eta > 0$
    \begin{equation} \label{eq;S_n-Pre0.3}
    \# \left\{ k: S^{(0)}_k \leqslant \eta n , \xi_k \geqslant (\log n )^\gamma  \right \} \stackrel{P}{\longrightarrow} \infty, \quad n \to \infty.
    \end{equation}
    
    \item  Let $\{Z(t)\}_{t \in \bbr_+}$ and $\{Z^\ast(t)\}_{t \in
        \bbr_+}$ be given by \eqref{eq;betaSUB} and
      (\ref{eq;shiftSUB0.1}), respectively. Then with $\vartheta_n$
      given by (\ref{eq;vartheta_n}), 
    \begin{align}
       &  \left\{ \frac{1}{\vartheta_n} S^{(0)\leftarrow} _ {\lfloor nt \rfloor}  \right \}_{  t \in \bbr_+  } 
       \Rightarrow \{  Z^\leftarrow (t) \}_{t\in \bbr_+ } \quad \text{in } \big( D(\bbr_+), J_1\big) ,   \label{eq;S_n-Pre0.4} \\
       & \left\{ \frac{1}{\vartheta_n}  S^{(n)\leftarrow} _ {\lfloor nt \rfloor}  \right \}_{  t \in \bbr_+  } 
       \Rightarrow \{  Z^{\ast\leftarrow} (t) \}_{t\in \bbr_+ }
        \quad \text{in } \big( D(\bbr_+), J_1\big).
       \label{eq;S_n-Pre0.5}
    \end{align}

\end{enumerate}

\end{proposition}

\begin{proof}
$(\rom1)$: The claim follows from Theorem 5.4 in \cite{samorodnitsky:wang:2019}.

$(\rom2)$: See (A.2) and Lemma A.1 in Appendix A in \cite{chen:samorodnitsky:2020}.

$(\rom3)$: See Lemma A.2 in Appendix A in \cite{chen:samorodnitsky:2020}.

$(\rom4)$:  We first show (\ref{eq;S_n-Pre0.4}). 
For a sequence $\{c_n\}$ satisfying  $n \overline{F}(c_n)
\sim  1/\Gamma(1-\beta)$ we have 
  \begin{equation} \label{eq;S_n-Pre1.0}
 \left\{  \frac{1}{c_n} S^ {(0)} _ {\lfloor nt \rfloor} \right \}_{t\in \bbr_+ } \Rightarrow
        \{Z(t)\}_{t \in \bbr_+} \quad \text{in } \big(D(\bbr_+),J_1
        \big); 
\end{equation}
see $e.g.$ Theorem 4.5.3 in \cite{whitt:2002}, and it follows that (\ref{eq;S_n-Pre0.4}) holds
in the $M_1$-topology (see \cite{whitt:1971a}). 
Because the  process $\{  Z^\leftarrow (t) \}_{t\in \bbr_+ }$ is $a.s.$
continuous,  the convergence also holds in the  $J_1$-topology, see
Section 12.4  in \cite{whitt:2002}. We combine this argument with part (i) of the
proposition to get (\ref{eq;S_n-Pre0.5}). 
\end{proof}

\medskip

We will occasionally drop the superscript on our random walk whenever the
discussion depends only on the step distribution of the
walk. Denoting for $A\subset \bbz$ and $a\in A$
$$
\PP_a \{\mathbf{S} \text{ escapes } A \}  = \PP \{ \mathbf{S} \text{ does not hit } A \setminus \{a\} \, \big\vert \, S_0 = a \},$$
we define the capacity of $A$ by 
\begin{equation} \label{eq;capDEF0.0}
    \text{cap} (A)  = \sum_{a \in A} \PP_a \{\mathbf{S} \text{ escapes
    } A  \}. 
\end{equation}
It is well known that $\text{cap}(A_1) \leqslant \text{cap}(A_2)$ if $A_1
\subset A_2$, and $\text{cap}(A_1
  \cup A_2) \leqslant \text{cap}(A_1) + \text{cap}(A_2)$ for any
  $A_1,A_2 $; see \cite{spitzer:1964}.

For $0\leqslant m_1 < m_2 \leqslant \infty$ we consider the range 
$$
A_0(m_1, m_2)  = \{ S^{(0)}_{m_1},\ldots, S^{(0)}_{m_2 -1} \}, 
$$
so that $A_0=A_0(0,\infty)$ is the full range.

\begin{proposition}  \label{prop;S_0}
 \begin{enumerate} [label=(\roman*)]
    \item  Fix $\gamma > (1-2\beta)^{-1}$, and let $\mathbf{S}=\{ S_t\}_{t\in \bbn_0}$ be
      independent of $A_0$. Then
      \begin{equation} \label{e:esc.cond}
\max_{0\leq j\leq n-(\log n)^\gamma} \PP \Bigl( \mathbf{S} \ \text{ hits } \
A_0 \cap \bigl\{ j+\lceil (\log n)^\gamma\rceil, \ldots, n\bigr\}\big|
S_0=j, A_0 \Bigr) \to 0
\end{equation}
$a.s.$ as $n\to\infty$. In particular, if for $n=1,2,\ldots, V_{1;n}$
and $V_{1;n}$ are measurable nonempty subsets of 
$A_0 \cap \{0,\ldots, n\}$
with $\min V_{2;n} - \max V_{1;n} \geqslant (\log n )^\gamma$, then  
     \begin{equation} \label{eq;gap0.1}
    \frac{1}{ \# V_{1;n} }\Bigl( \text{cap}\left( V_{1;n} \cup V_{2;n} \right) 
    - \text{cap}\left( V_{1;n}  \right) -
      \text{cap}\left( V_{2;n} \right)\Bigr)\to 0 \ \ \text{a.s..}
    \end{equation}

    \item Let $\widetilde{\mathbf{S}}^{(0)}$ be an independent copy of
      $\mathbf{S}^{(0)}$, with ranges denoted by
      $\widetilde{A}_0(\cdot,\cdot)$. Then  
      
    \begin{equation} \label{eq;c_infty} 
    c_\infty : = \PP\left\{ A_0 \cap  \widetilde{A}_0   = \{ 0 \}   \right \} 
       \in  (0,1).   
    \end{equation}
    Furthermore, 
    \begin{equation} \label{eq;CAPcvg0.0}
    \frac{\text{cap}\left(A_0(0,n) \right)}{n} 
        \longrightarrow c_\infty    \ \ \text{a.s.  as  $n \to \infty$.}
    \end{equation}
\end{enumerate}
\end{proposition}

\begin{proof}

  $(\rom1)$:  For \eqref{e:esc.cond} we write for $0\leq j\leq n-(\log
  n)^\gamma$, 
  \begin{align*}
&P\Bigl( \mathbf{S} \ \text{ hits } \
 A_0 \cap \bigl\{ j+\lceil (\log n)^\gamma\rceil, \ldots, n\bigr\}\big|
    S_0=j, A_0 \Bigr) \\
    \leqslant& 
  \sum_{k=\lfloor \log_2 \lceil(\log n)^\gamma\rceil  \rfloor} ^ { \lceil \log_2 n \rceil }
  \# \Big(  A_0  \cap [ j + 2^k , j+ 2^{k+1}  )    \Big)  
  \cdot \max_{m\geq 2^k}\PP_{0} \left\{  \mathbf{S} \text{ hits } m
               \right \}. 
  \end{align*}
  
  By  \eqref{eq;S_n-Pre0.2},   there is an $a.s.$ finite constant
   $B_1 = B_1(A_0)$
   such that the first term in the sum does not exceed
 $B_1 \log n/\bar F(2^k)$, while by \eqref{eq;S_n-Pre0.1}, the
  second term does not exceed $B_2 2^{(\beta - 1)k}/L(2^k)$ for some
  finite constant $B_2$. Therefore, for any $\epsilon>0$, the sum above
  can be bounded by
  $$
B_1  B_2  \sum_{k=\lfloor \log_2 (\log n)^\gamma  \rfloor} ^ \infty
  \frac{\log n}{\overline{F}(2^k)} \cdot 
   \frac{2^{(\beta - 1)k}}{L(2^k)} \lesssim (\log n) ^
   {1+(2\beta-1+\epsilon)\gamma}. 
   $$
Choosing $0<\epsilon<1-2\beta - \gamma^{-1}$ proves
\eqref{e:esc.cond}.

For  \eqref{eq;gap0.1} we enumerate, 
for a fixed $n$ and $i=1,2$,  $V_{i;n}$ from
left to right as $$ \{v_{i,1},\ldots,v_{i,n_i}\}.$$ 
Then 
\begin{align*}
   & \text{cap}\left( V_{1;n} \cup V_{2;n} \right)
    =   \sum_{j=1}^{n_1} \PP_{v_{1,j}} \left\{ \mathbf{S} \text{
     escapes } V_{1;n} \cup V_{2;n} \big|  A_0 \right \} 
      + \text{cap}\left( V_{2;n} \right) \\
   & \phantom{\text{cap}\left( V_{1;n} \cup V_{2;n} \right)}
   = \text{cap}\left( V_{1;n} \right) - 
     \sum_{j=1}^{n_1} q_j  +  \text{cap}\left( V_{2;n} \right) ,
\end{align*}
where in the obvious notation
\begin{align*}
   & q_j : = \PP\left\{ \mathbf{S} \text{ escapes } V_{1;n}
     \text{ but hits } V_{2;n} \big|  A_0, S_0=v_{1,j}\right \}.    
\end{align*}
Now  (\ref{eq;gap0.1}) follows from \eqref{e:esc.cond}.

$(\rom2)$:  Note that by \eqref{eq;S_n-Pre0.1}, 
$$
\PP \{ k \in   A_0  \cap \widetilde{A}_0   \} = \PP \{ k \in
 A_0  \} \cdot  \PP \{ k \in \widetilde{A}_0   \} \in
\text{RV}_{2\beta - 2}, 
$$
so it is a summable sequence. This implies (\ref{eq;c_infty}). 
To prove (\ref{eq;CAPcvg0.0}),  we  observe that the array 
$\{\text{cap}\left(A_0(m_1,m_2) \right)\}_{m_1 <m_2}$ forms a
stationary and subadditive family, so by the subadditive ergodic
theorem (see Theorem 5 in \cite{kingman:1968}) we have 
$$
\frac{\text{cap}\left( A_0(0,n) \right)}{n} \rightarrow \Upsilon \ \
a.s. 
$$
for some random variable $0\leqslant \Upsilon \leqslant 1$. Since the
invariant $\sigma$-field associated with the array is clearly 
trivial, it follows from Theorem 3 {\it ibid.} that $\Upsilon$ is
 a constant. It remains to show that the constant is equal to
$c_\infty$.  We 
have 
$$
\Upsilon =  \lim_{n\to \infty} \frac{1}{n} \sum_{i=0} ^ n \mathbb{P} \left\{ \widetilde{A}_0  +S^{(0)}_i \bigcap A_0(i,n) = \{i\} \right\}
$$
$a.s.$, hence
$$
\Upsilon =  \lim_{n\to \infty} \frac{1}{n} \sum_{i=0} ^ n \mathbb{P}
\left\{  \widetilde{A}_0  \cap A_0(0,n-i) = \{0\} \right\} \ \ a.s.
$$
This is the arithmetic mean of a sequence that converges to
$c_\infty$, so $\Upsilon =c_\infty$ $a.s.$.  
\end{proof}

\begin{proposition}  \label{prop;S_n}
For $n\in \bbn$,  let  
$$
    \overline{A_n}  =  \{ S^{(n)}_t : t =0,1,2,\ldots \}  \cap \{0,\ldots,n\}.
$$
Let $\{ Z^\ast (t)\}_{t\in \bbr_+}$, $\overline{R^\ast}$ and
$\vartheta_n$ be as in (\ref{eq;shiftSUB0.1}), 
(\ref{eq;barSRS}) and \eqref{eq;vartheta_n} respectively. 
\begin{enumerate}[label=(\roman*)]
    \item As $n\to \infty$
$$
     \left(  \left\{ \frac{1}{\vartheta_n} S^ { (n)  \leftarrow}_ {\lfloor nt \rfloor} \right\}_{t \in \reals_+ }, \frac{\overline{A_n}}{n}  \right) 
    \Rightarrow \left(  \{ Z^{ \ast \leftarrow} (t) \}_{t\in \reals_+} ,
    \overline{R^\ast}\right) \quad \text{in } \big(D(\bbr_+), J_1
  \big) \times \calF([0,1]).
$$

    \item  As $n \to \infty$
 \begin{align*}
   &\left\{  \frac{\text{cap} \left( \overline{ A_n } \cap \{ \lceil nx \rceil,\ldots, \lfloor ny \rfloor  \} \right)}
    { \vartheta_n }, \, 0\leqslant x<y \leqslant 1\right\} \\
\Rightarrow &\left\{ c_\infty \left( Z^{\ast\leftarrow}(y) -
    Z^{\ast\leftarrow}(x) \right),  \, 0\leqslant x<y \leqslant
              1\right\}
 \end{align*}
 in finite-dimensional distributions. 
\end{enumerate}
\end{proposition}

\begin{proof}
$(\rom1)$: 
Since the marginal convergence has been established in Proposition
\ref{prop;S_n-Pre} $(\rom1)$ and $(\rom4)$, the tightness is automatic. 
It suffices, therefore, to show  uniqueness of subsequential weak
limits. Suppose that for some subsequence $\{n_k \}$, 
$$
   \left( \frac{1}{\vartheta_{n_k}} \left\{ S^{(n_k) ^ \leftarrow }_ {\lfloor n_k t \rfloor} \right\}_{t \in \reals_+ }, \frac{1}{n_k}\overline{A_{n_k}}  \right) 
    \Rightarrow \left(  \{ Z^\kappa (t) \}_{t\in \reals_+},  
    R^\kappa \right) 
  $$
for some $Z^\kappa$ and $R^\kappa$. We will prove that, necessarily, 
\begin{equation} \label{eq;S_n2.1}
   \left(  \{ Z^\kappa (t) \}_{t\in \reals_+}, 
    R^\kappa \right) \eid \left(  \{ Z^{ \ast \leftarrow} (t) \}_{t\in \reals}, 
    \overline{R^\ast}\right).
\end{equation}
To this end, recall that the $\pi$-systems 
$$
   \mathcal{C}_D  = \big\{  \{ x(t_i) > a_i : t_i \geqslant 0, a_i \in \bbr, 1\leqslant i \leqslant
   \ell  \} : \ell \in \bbn   \big \}
   $$
   and
$$   
   \mathcal{C}_\mathcal{F}  = \left\{ \mathcal{F}^T  : T \ \ \text{a
       finite union of open intervals}      \right\} 
$$
generate the respective $\sigma$-fields in $D(\bbr_+)$  and
$\mathcal{F}$, so it is enough to check that the 
laws of $ \left(  \{ Z^ \kappa (t) \}_{t\in \reals_+}, 
    R^\kappa \right)$ and 
    $ \left(  \{ Z^{ \ast } (t) \}_{t\in \reals_+}, \overline{R^\ast}\right)$
agree on $\mathcal{C}_D\times \mathcal{C}_\mathcal{F}$. That is, for any $\ell_1,\ell_2 \in \bbn_0,\,
t_i> 0, a_i \in \bbr$, $i=1,\ldots, \ell_1$ and disjoint open intervals  $T_1,\ldots,
T_{\ell_2}$, we have
$$
    \PP \Bigl\{ \bigcap_{\substack{1\leqslant i \leqslant \ell_1 \\ 1\leqslant j \leqslant \ell_2}} \left\{ R^\kappa \cap T_j \neq \emptyset ,  Z^\kappa (t_i) > a_i \right\} \Bigr\}
   =
    \PP \Bigl\{
    \bigcap_{\substack{1\leqslant i \leqslant \ell_1 \\ 1\leqslant j \leqslant \ell_2}}  \{ \overline{R^\ast} \cap T_j \neq \emptyset ,  Z^{\ast \leftarrow} (t_i) > a_i  \} \Bigr\}.
    $$
Denote $T_1 = (c_1,d_1), \ldots, T_{\ell_2} =
(c_{\ell_2},d_{\ell_2})$. We have 
\begin{align*}
     & \PP \Big \{ \bigcap_{\substack{1\leqslant i \leqslant \ell_1 \\ 1\leqslant j \leqslant \ell_2}}
   \left\{   Z^\kappa (t_i) > a_i , R^\kappa \cap T_j \neq \emptyset \right\} \Big \}  \\
   = &  \lim_{k \to \infty } \PP \bigg\{ \bigcap_{\substack{1\leqslant i \leqslant \ell_1 \\ 1\leqslant j \leqslant \ell_2}}
   \bigg \{   \frac{S ^ {( n_k) ^ \leftarrow} _ { \lfloor n_k t_i \rfloor }}{\vartheta_{n_k}}  > a_i ,
   \overline{A_{n_k}} \cap n_kT_j  \neq \emptyset 
   \bigg \} \bigg \} \\
   = &   \lim_{k \to \infty } \PP \bigg\{ \bigcap_{\substack{1\leqslant i \leqslant \ell_1 \\ 1\leqslant j \leqslant \ell_2}}
   \bigg \{   \frac{ S^  {( n_k) ^ \leftarrow} _ { \lfloor n_k t_i \rfloor }}{\vartheta_{n_k}} > a_i , 
    \frac{ S^ {(n_k) \leftarrow} _{ \lfloor n_k d_j \rfloor }
   - S^ {(n_k) \leftarrow} _ { \lceil n_k c_j \rceil }}
   {\vartheta_{n_k}} > 0 
   \bigg \} \bigg \} \\
   = & \PP \Big\{ \bigcap_{\substack{1\leqslant i \leqslant \ell_1 \\ 1\leqslant j \leqslant \ell_2}} \left\{  Z^{\ast\leftarrow} (t_i) > a_i ,  Z^{\ast \leftarrow} (d_j) -  Z^{\ast\leftarrow} (c_j) > 0   \right \} \Big \} \\
   = & \PP \bigg\{ \bigcap_{\substack{1\leqslant i \leqslant \ell_1 \\ 1\leqslant j \leqslant \ell_2}} \left\{   Z^{\ast\leftarrow} (t_i) > a_i , \overline{R^\ast} \cap T_j \neq \emptyset  \right\} \bigg \},
\end{align*}
as long as we can justify the penultimate equality. 

Since each
$Z^{\ast\leftarrow} (t_i)$ is a continuous random variable,
by  the
Portmanteau Theorem  we only
need to check that
$$
\lim_{\vep\to 0} \limsup_{n\to\infty}\PP \bigg\{  
   0< \frac{ S^ {(n) \leftarrow} _{ \lfloor n d \rfloor }
   - S^ {(n) \leftarrow} _ { \lceil n c \rceil }}
   {\vartheta_n}
   <\vep
   \bigg \}  =0
$$
for any $0\leq c<d\leq 1$. This follows from the
marginal
convergence given in Proposition \ref{prop;S_n-Pre} $(\rom1)$.

$(\rom2)$: By the Skorokhod embedding theorem, the convergence in part
$(\rom1)$ of the proposition  holds as the $a.s.$ 
convergence on some probability space, which will again be denoted by $(\Omega,\mathscr{F},\PP)$  for
typographical convenience. 
Consider the partition $\Omega=\Omega_1 \cup \Omega_2$ with 
$$
 \Omega_1 : = \left \{ \omega \in  \Omega : \overline{R^\ast}(\omega) \cap (x,y) 
 = \emptyset \right \} ,  \quad
 \Omega_2 : = \left \{ \omega \in  \Omega : \overline{R^\ast}(\omega) \cap (x,y) 
  \neq \emptyset \right \} .   
$$
We will show that the required convergence holds  in probability  on both $\Omega_1$
and $\Omega_2$.

Since $\overline{R^\ast}$ does not hit fixed points, we have
$Z^{\ast\leftarrow}(y) - Z^{\ast\leftarrow}(x) =0$ $a.s.$ on
$\Omega_1$. Furthermore, we can write for
some null event $\Omega_{0,1}$, 
$$
   \Omega_1 = \Omega_{0,1}\cup \bigcup_{k\geqslant 1}
   \Omega_1^k, \quad \Omega_1^k: =
   \left\{ \rho\left( \overline{R^\ast}(\omega) , [x,y]  \right) \geqslant 1/k \right \},
$$
where $\rho$ is defined in (\ref{eq;ssMtc}). Since
$\frac{1}{n} \overline{A_n}  \longrightarrow \overline{R^s}$ a.s. in
the Fell topology, the convergence holds  
in the Hausdorff metric $\rho_{\text{H}}$ as well, so on each 
$\Omega_1^k$,
$$
 \overline{A_n} \cap \{ \lceil nx \rceil ,\ldots, \lfloor ny \rfloor \} = \emptyset
 $$
for all $n$ large enough. Hence the required convergence  holds $a.s.$ 
on $\Omega_1$. 

We now consider the event $\Omega_2$. Let  
$$
\tau_1  =  \inf \left\{ \overline{R^\ast} \cap [x,y] \right \}, \ \
\tau_2 =\sup \left\{ \overline{R^\ast} \cap [\tau_1,y] \right \} \ \ 
(\text{both} \ =y \ \
 \text{if} \ \ \overline{R^\ast} \cap [x,y]=\emptyset). 
$$
For some null event $\Omega_{0,2}$ we can write 
$$
   \Omega_2 = \Omega_{0,2}\cup \bigcup_{k\geqslant 1}
   \Omega_2^k, \quad \Omega_2^k: =
   \left\{  \tau_2(\omega)-  \tau_1(\omega)  \geqslant 1/k \right \}.
   $$
   On $\Omega_2$, by the strong Markov property,
   $$
   \text{cap} \Big(
   \overline{A_n} \cap 
   \{ \lceil nx \rceil ,\ldots, \lfloor ny \rfloor \} \Big) \eid
   \text{cap} \Big(
   A_0\Bigl(
   0,  S ^{(n)\leftarrow}_{\lfloor ny \rfloor} -
   S^{(n)\leftarrow}_{\lceil nx \rceil}\Bigr) \Big).
   $$
  Once again, since  $\frac{1}{n} \overline{A_n}  \longrightarrow
  \overline{R^s}$ $a.s.$ in the Hausdorf metric. So on 
  each $\Omega_2^k$, we have
 $S ^{(n)\leftarrow}_{\lfloor ny \rfloor} -
   S^{(n)\leftarrow}_{\lceil nx \rceil}\to \infty$ $a.s.$. It follows by Proposition
   \ref{prop;S_0} $(\rom2)$ that
$$
   \frac
   {
     \text{cap}\left( \overline{A_n} \cap 
   \{ \lceil nx \rceil ,\ldots, \lfloor ny \rfloor \}  \right) 
   }
   {
   S ^{(n)\leftarrow}_{\lfloor ny \rfloor} - S^{(n)\leftarrow}_{\lceil nx \rceil}
   }  \longrightarrow c_\infty
   $$
   in probability on each $\Omega_2^k$, hence also on the entire
   $\Omega_2$. Finally,
   \begin{align*}
&\frac{\text{cap} \left( \overline{ A_n } \cap \{ \lceil nx \rceil,\ldots, \lfloor ny \rfloor  \} \right)}
 { \vartheta_n } \\
 = &\frac{
     \text{cap}\left( \overline{A_n} \cap 
   \{ \lceil nx \rceil ,\ldots, \lfloor ny \rfloor \}  \right)}
   {S ^{(n)\leftarrow}_{\lfloor ny \rfloor} - S^{(n)\leftarrow}_{\lceil
     nx \rceil}}
   \frac{S ^{(n)\leftarrow}_{\lfloor ny \rfloor} - S^{(n)\leftarrow}_{\lceil
     nx \rceil}}{ \vartheta_n } \to c_\infty \left(
     Z^{\ast\leftarrow}(y) - Z^{\ast\leftarrow}(x) \right) 
   \end{align*}
   in probability on $\Omega_2$. 
\end{proof}

We proceed with an important lemma. Switching back to the terminology
of Subsection \ref{sec;MC}, we suppose that the random elements $\{
Y^{(k;n)} \}_{k\in \bbn}$ are defined on a probability space
$(\Omega,\mathscr{F},\PP)$, while $\{
Y^{(0;n)} \}$ is defined on a different probability space, and the
entire system is defined on the product probability space. We will use
the notation $\PP_\omega$ for the quenched (conditional) probability
(computed with respect to $\{Y^{(0;n)} \}$). When needed in the
sequel, the notion of quenched probability may change, and we will
always specify its precise meaning. 

\begin{lemma} \label{lem;ASYidp}
For any $K \in \bbn$ and $\epsilon > 0$ there exists a sequence of
events $\{ \Omega_{[K];n} ^ \epsilon \}_{n \geqslant 1}$ in $\Omega$
satisfying  
 \begin{equation} \label{eq;ASYidp0.0}
   \liminf_{n\to \infty}\PP \left \{  \Omega_{[K];n} ^ \epsilon \right
   \} >1- \epsilon, 
    \end{equation}
with the following properties.  
\begin{enumerate}[label=(\roman*)]
    
    \item   There exists $C=C(\epsilon)>0$ such that 
for all $1\leqslant k \leqslant K$  and all large $n$,  
    \begin{equation} \label{eq;ASYidp0.1}
    C^{-1} \frac{\vartheta_n}{w_n} \leqslant \overline{p}_{k;n} \leqslant C \frac{\vartheta_n}{w_n}
    \quad  \text{on } \Omega_{[K];n} ^ \epsilon . 
    \end{equation}
    
    \item  For all $1\leqslant k_1 \neq k_2 \leqslant K$, $I_{k_1;n} \cap I_{k_2;n} = \emptyset$ on $\Omega^\epsilon_{[K];n}$.

    \item  For $1\leqslant k_1 \leqslant k_2 \leqslant K$  and
      $n\geqslant 1$,  set  
    $$\overline{p}_{[k_1:k_2];n} = 
     \PP \left\{  I_{0;n} \cap ( I_{k_1;n} \cup \cdots \cup I_{k_2;n} ) \neq \emptyset 
     \, \big| \,  I_{k_1;n}  \ldots  I_{k_2;n}   \right \}. 
    $$ 
    Then
    \begin{equation} \label{eq;ASYidp0.2}
    \overline{p}_{[k_1:k_2];n} = \sum_{k=k_1}^{k_2} \overline{p}_{k;n} - o \left( \frac{\vartheta_n}{w_n} \right) \quad  \text{on } \Omega_{[K];n} ^ \epsilon . 
    \end{equation}
\end{enumerate}
\end{lemma}

\begin{proof}
The Skorohod embedding argument we have just
used shows that there is a probability space (once again denoted by
$(\Omega,\mathscr{F},\PP)$) on which, for each $1\leqslant k\leqslant
K$, 
\begin{equation}\label{eq;couple0.1}
\left(  \left\{ \frac{1}{\vartheta_n} S^ { (k;n)  \leftarrow}_
    {\lfloor nt \rfloor} \right\}_{t \in \reals_+ }, \, \frac{1}{n}I_{k;n}  \right) 
   \longrightarrow \left(  \{ Z^{ \ast \leftarrow}_k
      (t) \}_{t\in \reals_+} ,\, 
    \overline{R^\ast _k}\right)    
\end{equation}
$a.s.$ in $\big( D(\bbr_+) , J_1 \big) \times  \calF([0,1])$  and 
\begin{equation} \label{eq;couple0.2}
\frac{\text{cap} \left( I_{k;n} \cap \{ \lceil nx \rceil,\ldots, \lfloor ny \rfloor  \} \right)}
    { \vartheta_n }    \longrightarrow c_\infty \left( Z^{\ast \leftarrow}_k(y) - Z^{\ast \leftarrow}_k(x) \right)    
  \end{equation}
in probability for all $0\leqslant x \leqslant y \leqslant 1$.  In the
remainder of the proof we work on this probability space. We
spell out the argument in the case $K=2$;  the general case 
can be treated similarly.  

$(\rom1)$: The last visit decomposition shows that 
$$
 \overline{p}_{k;n} =   \text{cap}(I_{k;n}) \big/ w_n, \quad k=1,2.
 $$
For $\vep>0$ choose $C_1>0$ so large that $\PP \left(
  C_1^{-1}\leqslant Z^{\ast
    \leftarrow}(1)   \leqslant C_1\right)> 1-\epsilon/2$. Letting
$$
\Omega_{[2];n}^\epsilon =\bigl\{ c_\infty C_1^{-1}\leqslant
\text{cap} ( I_{k;n})/\vartheta_n \leqslant  c_\infty C_1, \,
k=1,2\bigr\}, 
$$
we see by \eqref{eq;couple0.2} with $x=0,y=1$ that
\eqref{eq;ASYidp0.0} holds. Then  (\ref{eq;ASYidp0.1}) holds with
$C=C_1/c_\infty$. 

$(\rom2)$:  Since $\PP ( I_{1;n}  \cap I_{2;n} \neq \emptyset ) \to
0$, we can make the events   
$\Omega^\epsilon_{[2];n}$ slightly smaller so 
that  (\ref{eq;ASYidp0.0}) still holds and the condition of $(\rom2)$
also holds. 

$(\rom3)$: Recall that $\overline {R^\ast _1}$ and $\overline {R^\ast
  _2}$ intersect only on a null set. It follows from
\eqref{eq;couple0.1} that $\liminf 
       \rho\left( I_{1;n}, I_{2;n}    \right)/n>0$ a.s. on
    $\Omega^\epsilon_{[2];n}$. After removing from
    $\Omega^\epsilon_{[2];n}$ the null set, 
The claim \eqref{eq;ASYidp0.2} now  
    follows from \eqref{e:esc.cond}, where the initial point $j$ is
    the leftmost point in $I_{1;n} \cup I_{2;n}$ that is in $I_{0;n}$,
    and $A_0$ the extension to the left of that among $I_{1;n},
    I_{2;n}$ which does not contain $j$. 
 \end{proof}

\begin{lemma} \label{lem;ASYdtc}
For any $ K, m \in \mathbb{N}$ , we have
\begin{equation} \label{eq;ASYdtc0.0}
    \lim_{n\to \infty}
    \PP \left\{  \text{the numbers} \  \left\{ j_{k,i;n}
      \right\}_{\substack{1\leqslant k \leqslant K \\ 1\leqslant i
          \leqslant m}} \text{ are all different} \right \} = 1 .
\end{equation}
\end{lemma}
\begin{proof}
Again, we only spell out the argument in the case $K=2$ and $m=1$.
For $0<\epsilon<1$ let $\Omega_{[K];n} ^ \epsilon $ and
$C=C(\epsilon)>0$ be as in Lemma \ref{lem;ASYidp}. We have for a large $a>0$ 
\begin{align*}
    & \PP \left\{ j_{1,1;n} = j_{2,1;n} \right\} \\
    \leqslant &  \PP \left\{ j_{1,1;n} = j_{2,1;n} \leqslant aC \frac{w_n}{\vartheta_n}  
    \right\}   + \PP \left\{  j_{1,1;n}  > aC \frac{w_n}{\vartheta_n}
    \right\}.
\end{align*}

Letting now $\PP_\omega$ be the quenched probability given
$\{Y^{(1;n)} \}$, we recall that, with respect to $\PP_\omega$, $j_{1,1;n}$ is
geometrically distributed with success probability
$\overline{p}_{1;n}$. Therefore, 
\begin{align*}
  &\PP \left\{  j_{1,1;n}  > aC \frac{w_n}{\vartheta_n} \right\}=  \int_\Omega \PP_\omega \left\{ 
    j_{1,1;n} > aC \frac{w_n}{\vartheta_n}
    \right \} 
    \PP( d \omega ) \\
 \leq & \epsilon + \int_{\Omega_{[K];n} ^ \epsilon} \PP_\omega \left\{
        j_{1,1;n} > aC \frac{w_n}{\vartheta_n}     \right \}     \PP(
        d \omega ) \\
  \leq & \epsilon + \left( 1-C^{-1} \frac{\vartheta_n}{w_n}
         \right)^{aCw_n/\vartheta_n-1}
         \to\epsilon+ e^{-a}.
\end{align*}

 On the other hand, by the inclusion-exclusion formula and Lemma
 \ref{lem;ASYidp} $(\rom3)$, 
\begin{align*}
&\PP \left\{ j_{1,1;n} = j_{2,1;n} \leqslant aC \frac{w_n}{\vartheta_n}  
                 \right\}   \\
\leqslant &\epsilon +  \frac{aCw_n}{\vartheta_n} \int_{\Omega^\epsilon_{[2];n}}
    \PP_\omega \left\{ I_{0;n} \cap I_{1;n} \neq \emptyset, \, 
     I_{0;n} \cap I_{2;n} \neq \emptyset
    \right \}
            \PP (d \omega) \\
 \leqslant &\epsilon +  \frac{aCw_n}{\vartheta_n}  \sup_{\omega \in
             \Omega^\epsilon_{[2];n} } \PP_\omega \left\{ I_{0;n} \cap I_{1;n} \neq \emptyset, \, 
     I_{0;n} \cap I_{2;n} \neq \emptyset
    \right \} \to\epsilon. 
\end{align*}
Letting first $\epsilon\to 0$ and then $a\to\infty$ concludes the
argument. 
\end{proof}

\begin{proof}[Proof of Theorem \ref{thm;4joint}]

For notational simplicity we consider the case $K=m=2$. Our method
easily carries over to arbitrary $K$ and $m$. We will once again use
the Skorohod embedding and assume that \eqref{eq;couple0.1} and
\eqref{eq;couple0.2} hold. Then $I_{k;n}/n\to \overline{R^*_k}  $
$a.s.$ and by Proposition \ref{prop;S_n} $(\rom2)$, $w_n
\overline{p}_{k;n}/\vartheta_n\to c_\infty Z^{*\leftarrow}_k (1)$
$a.s.$  as well. We consider now the remaining components in
\eqref{eq;4joint0.0}. 

Let $\PP_\omega$ be the quenched probability given
$\{Y^{(k;n)} \}, \, k=1,2$. To handle the second component in
\eqref{eq;4joint0.0}, it is enough to show  that, for $a.s.$ $\omega \in \Omega$,
\begin{equation} \label{e:Gammas}
     \left( j_{1,1;n}  \overline{p}_{1;n},  j_{2,1;n}  \overline{p}_{2;n}
 \right)     \Rightarrow (\Gamma_{1,1}, \Gamma_{2,1}) 
\end{equation}
under $\PP_\omega$. For $0<\epsilon<1$, let $\Omega^\epsilon_{[2];n}$ be
the event in Lemma \ref{lem;ASYidp}.  Let $x_1,x_2>0$. On  
$\Omega^\epsilon_{[2];n}$, for any subsequence $(n_m)$ over which
$\left(\overline{p}_{2;n_m}  \right) ^{-1}  x_2
 \geqslant \left(\overline{p}_{1;n_m}  \right) ^{-1}  x_1 $ we have 
\begin{align*}
    & \PP_{\omega}  \left\{ j_{1,1;n_m} \overline{p}_{1;n_m} \geqslant x_1,\,  j_{2,1;n_m} \overline{p}_{2;n_m}  \geqslant x_2    \right \} \\
    = &  \left(  1 - \overline{p}_{ [1:2];n_m } \right) 
    ^ { \lfloor (\overline{p}_{1;n_m})^{-1} x_1\rfloor -  1}  \cdot 
     \left(  1 - \overline{p}_{2;n_m } \right) 
    ^ { \lfloor (\overline{p}_{2;n_m})^{-1} x_2-  (\overline{p}_{1;n_m})^{-1} x_1\rfloor } \\
   = & 
   \left( 
   \prod_{k=1}^2( 1- \overline{p}_{k;n_m}) + o\left( \frac{\vartheta_{n_m}}{w_{n_m}} \right)
   \right)
   ^ {\lfloor (\overline{p}_{1;n_m})^{-1} x_1\rfloor -   1} \\
   & \cdot 
    \left(  1 - \overline{p}_{2;n_m } \right) 
    ^ { \lfloor(\overline{p}_{2;n_m})^{-1} x_2 - (\overline{p}_{1;n_m})^{-1} x_1 \rfloor}
   \\
   = & 
   \bigl( 1+o(1)\bigr)
   \prod_{k=1}^{2} 
   \left( 
     1- \overline{p}_{k;n_m}  + o\left( \frac{\vartheta_{n_m}}{w_{n_m}} \right)
   \right)
   ^ { (\overline{p}_{k;n_m})^{-1} x_k }\to e^{-(x_1+x_2)}
\end{align*}
as $m\to \infty$.  The same is true for any subsequence $(n_m)$ over which
$\left(\overline{p}_{2;n_m}  \right) ^{-1}  x_2
< \left(\overline{p}_{1;n_m}  \right) ^{-1}  x_1 $. We thus see that
$$
\PP_{\omega}  \left\{ j_{1,1;n_m} \overline{p}_{1;n_m} \geqslant
  x_1,\,  j_{2,1;n_m} \overline{p}_{2;n_m}  \geqslant x_2    \right \}
\to e^{-(x_1+x_2)}
$$
over $\liminf \Omega^\epsilon_{[2];n}$ for every $0<\epsilon<1$ and,
hence, also on an event of probability 1. Since this is true for all
$x_1,x_2>0$, \eqref{e:Gammas} follows. 

We now consider the last component in \eqref{eq;4joint0.0}. By Lemma
\ref{lem;ASYdtc} we only need to  prove the following
statement. Consider $\ell$ disjoint open 
intervals in $(0,1)$,  $\{ B_i = (x_i,y_i):i=1,\ldots,\ell\}$. Then
for any  $\epsilon , \delta >0$, there exists a sequence of events $\{
\Omega^\prime_n \}_{n \in \bbn} $ in $\Omega$ such that  
     $\liminf_{n\to \infty} \PP\{\Omega_n ^ \prime \} \geqslant 1 - \epsilon$ and 
\begin{equation} \label{eq;4joint2.1}
   \sup_{ \omega \in \Omega^ \prime _n  }
  \left|  \PP_{\omega} \left \{ \bigcap_{r=1}^\ell \left\{ \frac{1}{n} I_{1,1;n} \cap B_r \neq \emptyset \right \} \right\} - \PP_{1} \left\{ \bigcap_{r=1}^\ell \left\{ J_{1,1} \in B_r  \right \} \right \} \right| \leqslant \delta  
\end{equation}
where $\PP_\omega$ is the quenched probability given $Y^{(1;n)}$ and
$\PP_1$ is the probability associated with an independent standard
uniform random variable.  We treat the cases $\ell = 1$ and $\ell
\geqslant 2$ separately.

Suppose first that  $\ell = 1$. By \eqref{eq;HdfMea*}, 
$$
\PP_{1} \left\{ {J_{1,1}} \in  B_1  \right \}  = 
\frac{
Z^{\ast \leftarrow}_1 (y_1) - Z^{\ast \leftarrow}_1 (x_1)   } {
Z^{\ast \leftarrow}_1 (1) },
$$
while by the last exit decomposition, 
$$
 \PP_{\omega} \left\{  \frac{I_{1,1;n}}{n}  \cap B_1 \neq \emptyset  \right \} 
    =  \frac{\text{cap}\bigl( I_{1;n}(\omega) \cap (nx_1, ny_1) \bigr)}{\text{cap}(
      I_{1;n} (\omega) )}  .
    $$
Therefore, we can take $\Omega^\prime_n = \Omega$ and 
(\ref{eq;4joint2.1}) follows by \eqref{eq;couple0.2}. 

If $\ell \geqslant 2$,  then the second probability in
(\ref{eq;4joint2.1}) vanishes. Furthermore, 
\begin{align*}
    & \PP_\omega  \left\{ I_{1,1;n}\cap n B_1 \neq \emptyset, \ldots, 
    I_{1,1;n}\cap B_\ell \neq \emptyset \right \} \\
    \leqslant & \PP_\omega \left\{ I_{1,1;n}\cap n B_1 \neq \emptyset, 
    I_{1,1;n}\cap n B_2 \neq \emptyset \right \}. 
\end{align*}
Letting $\delta =
x_2-y_1>0$ we have by the strong Markov property,
\begin{align*}
    & \PP \left\{ I_{1,1;n}\cap n B_1 \neq \emptyset, 
      I_{1,1;n}\cap n B_2 \neq \emptyset \right \} \\
     \leqslant &  \PP \left\{ A_0, \widetilde{A}_0 \ \
                 \text{have a common point} >n\delta \right\} \to 0
\end{align*}
as $n\to\infty$ by \eqref{eq;S_n-Pre0.1}, we immediately obtain
(\ref{eq;4joint2.1}). 
\end{proof}

\begin{proof}[Proof of Proposition \ref{prop;1,1card}]

$(\rom1)$: The claim (\ref{eq;1,1card0.1}) follows from the obvious
fact that $$ \# I_{1,1;n}\leqslant\#(A_0 \cap
\widetilde{A}_0)$$ and the latter cardinality 
has the geometric distribution with success probability
$1-c_\infty$ (with $c_\infty$ defined in \eqref{eq;c_infty}).

$(\rom2)$: The claim follows from $(\rom1)$ by the Markov inequality.
\end{proof}

\section{Additional auxiliary results}  \label{sec;AppB}

This section contains several auxiliary results that are essential in the main proofs.  
We start with
describing certain useful properties of the functions $V$ and $h$ in
\eqref{eq;aux-h}  and some related functions. Let 
\begin{align} 
    &  G(x)  = \left( 1 \Big/  \gamma \overline{H} \right) ^
      \leftarrow (x),  \quad x \geqslant 1/\gamma;   
    \label{eq;V&G}  
\end{align}
notice that by the inverse function theorem,
 \begin{equation} \label{eq;Gprime}
    x G^\prime(x)  = h \circ G(x).
  \end{equation}
Furthermore, by the Karamata theorem the function
$$
x\mapsto \int_1 ^ x \frac{du}{u^{1-\alpha} L_\alpha (u)}
$$
is regularly varying at infinity with exponent $\alpha$, so its inverse
is regularly varying with exponent $1/\alpha$. Therefore, the function
\begin{equation} \label{eq;scr(L)}
    \mathscr{L}(x) =
    x^{-1/\alpha}\left(  \int_1 ^ x \frac{du}{u^{1-\alpha} L_\alpha (u)} \right) ^ \leftarrow 
  \end{equation}
is slowly varying. 
\begin{proposition} \label{prop;VanGogh}
The functions $V$, $h$, $G$ and $\mathscr{L}$ have the following
properties at infinity:
\begin{align}
  & G(x) - V(x) = o\big( h \circ G(x)  \big).  \label{eq;asyV&G0.2} \\
    & V(x) \sim    G (x) \sim
    (\log x ) ^{ 1/\alpha } \mathscr{L} (\log x), 
    \label{eq;asyV&G0.0}  \\
    & h \circ V(x) \sim h \circ G(x) \sim \alpha^{-1} (\log x)^{1/\alpha - 1}
      \mathscr{L}(\log x)   \label{eq;asyV&G0.1}.
\end{align}
Furthermore, for every $t>0$ the function $V$ satisfies
\begin{equation} \label{e:pi.var}
\lim_{x\to\infty} \frac{V(tx)-V(x)}{h\circ V(x)} =\log t. 
\end{equation}
\end{proposition}
\begin{proof}
The statement (\ref{eq;asyV&G0.2}) follows from the properties of the
tails in the Gumbel domain of attraction;
see $e.g.$ (2.4) in \cite{chen:samorodnitsky:2020}. This now implies the
first asymptotic equivalencies in \eqref{eq;asyV&G0.0} and
\eqref{eq;asyV&G0.1}. Since
$$
G(x)=\bigl( \log(x\gamma) \bigr)^{1/\alpha} \mathscr{L}\bigl(
\log(x\gamma) \bigr),
$$
the second asymptotic equivalence in \eqref{eq;asyV&G0.0} follows from
the regular variation. Further, by Karamata's theorem,
$$
L_\alpha\bigl(G(x)\bigr)
\sim \bigl(\mathscr{L}(\log
x)\bigr)^\alpha/\alpha,
$$
and the second asymptotic equivalence in \eqref{eq;asyV&G0.1} follows
as well.

The version of the statement \eqref{e:pi.var} with $V$ replaced by $G$
follows easily from the definition of $G$, and by \eqref{eq;asyV&G0.2}
we may replace $G$ by $V$.
\end{proof}

We proceed with two lemmas used in the proof of Proposition
\ref{prop;bottom}. The first lemma is purely analytical, and we omit a
straightforward proof. 

\begin{lemma} \label{lem;psi}
\begin{enumerate}[label=(\roman*)]
\item The function 
$$
\psi(r) =  \frac{(1-\beta)^{1/\alpha} + \beta ^{1/\alpha}}{(1-\beta -
  r) ^{1/\alpha}}-1, \ 0\leq r<1-\beta
$$
is  increasing to infinity. Furthermore, the numbers $r_m$ defined by
$\psi(r_m)=m$ satisfy $r_m < m/(m+1)-\beta$  for $m\geqslant 1$. 
\item The function
  \begin{align*}
  &\widetilde{\psi}(r) =    (1-\beta - r) \bigl[ \lfloor \psi(r)
    \rfloor    + \bigl( \psi(r)- \lfloor \psi(r)
    \rfloor\bigr)^\alpha\bigr], \ 0\leq r<1-\beta 
    \end{align*}
 is increasing  and satisfies $\widetilde{\psi}(r) > r+\beta$ on $(0,1-\beta)$. Finally,
    $\widetilde{\psi}(r) \to \infty$ as  $r\to (1-\beta)^-$. 
\end{enumerate}
 \end{lemma}
The next lemma is essential for Proposition \ref{prop;bottom}. 
\begin{lemma} \label{lem;CvxAls}
Fix any $m \in \bbn_0$. 
\begin{enumerate} [label=(\roman*)]
    \item For  $1< b<\infty$ such that
      $\nu ( 1,b)>0$ let $\{
      \xi_i \}_{i \in \bbn}$ be   $i.i.d.$ 
    random variables whose law is the restriction of $\nu$ to 
    $(1, b) $ normalized to a probability
    measure. Then for any $m\in \bbn_0$  there are
    $c_m , \gamma_m \geqslant 0$  depending on $m$ only  such that for
    any     $y\in \bigl( mb ,  (m+1)b\bigr]$     and $d\geqslant m+1$,   
\begin{equation} \label{eq;CvxAls0.0}
    \PP \left\{ \sum_{i=1}^d \xi_i \geqslant y \right \}
    \leqslant c_m ^ d b^ {\gamma_m } \big( \overline{H} ( b)  \big)  ^ m  \overline{H} ( y - mb ) .
\end{equation}
\item For $0<r<1-\beta$ consider sequences 
$z_n = V(w_n) + V(\vartheta_n) - V(w_n  / n ^ r) - o \left(  V(w_n)
\right)$ and $\overline{z}_n = V(w_n  / n ^ r)$, $n\geqslant 1$. 
Let $(r_m)$ be as in Lemma \ref{lem;psi}. Then for any
  $m\geqslant 1$ there is $\gamma_m>0$ such that for any $r\in
  (r_m,r_{m+1}]$,  
\begin{equation} \label{eq;CvxAls0.1}
\lim_{n\to \infty}
\frac
{
\PP \left\{ \sum_{ \Gamma_j > n^r } 
V\left(  w_n  \big/ \Gamma_j \right) 1_{\{ 0 \in I_{j;n} \}   } \geqslant z_n
\right \}
}
{
\overline{z}_n ^ {\gamma_m} \big( \overline{H}(\overline{z}_n)   \big) ^ m   \overline{H}( z_n -m \overline{z}_n)
} =0. 
\end{equation}
 \end{enumerate}
\end{lemma}
\begin{proof} 
$(\rom1)$: Recall that
  $\overline{H}(x)=\exp\{-q(x)\}$ for $q\geq x_0$ with $q$ increasing
  and concave, and $xq^\prime(x)\leqslant q(x)$ for   $x\geqslant
  x_1$, for some $x_1\geqslant x_0$. Therefore, we
  can extend $q$  in the obvious way from the range $[x_1,\infty)$ 
  to an increasing and concave function on   $[0,\infty)$, that
  vanishes at the origin.   We work with this redefined
  $H$, while keeping the original notation $H$.  
There clearly is $C\geq 1$ so that
$\PP\{ \xi_1 > x \} \leqslant C \overline{H}(x)$ for all $x>0$. 
Since $\overline{H}$  is the tail of a subexponential
distribution, there is $c_0>0$ such that, in the usual notation for
the convolution power, for all $y>0$ 
\begin{equation} \label{eq;CvxAls1.0}
  \PP  \left\{ \sum_{i=1}^d  \xi_i \geqslant y \right \} \leqslant C^d
  \overline{H^{*d}} (y)    \leqslant  c_0 ^ d  \overline{H}(y) \quad
  \text{for all } d \in \bbn,  
\end{equation}
see  Proposition 4.1.10 in \cite{samorodnitsky:2016:SPLRD}. This gives
(\ref{eq;CvxAls0.0}) in the case of $m=0$ and all $d\geqslant 1$
(with $\gamma_0=0$).

We proceed in the inductive manner. Assume that (\ref{eq;CvxAls0.0})
holds for all 
  $0\leqslant m \leqslant  m_0$ and all $d\geq m + 1$. We first consider
  the case $m=m_0+1$ and $d=m+1$. Let $ H_b$ be the restriction of
  $H$ to $(1,b)$. We still have 
  $$\PP\{ \xi_1 > x \} \leqslant (C /\|H_b\|)
  \overline{ H_b}(x) \quad \text{for all } x>0.$$
  Therefore,  for $y>0$
 \begin{align*}
   \PP&\left\{ \sum_{i=1}^{d}  {\xi}_i \geqslant  y  \right\}
   \leqslant  (C /\|H_b\|)^d   \overline{H_b^{*d}} (y)  \\
    = & (C /\|H_b\|)^d \int_{(0,b)^{d}}
    1_{ \left\{ \sum_{i=1}^{d} z_i > y  \right\} }
    \prod_{i=1}^{d}
    \exp \{ - q(z_i) \} 
    q^\prime(z_i)    d z_i \\
    \leqslant &  (C /\|H_b\|)^d \bigl(q(b) \big) ^ {m+1} \\
    & \cdot 
    \exp\left\{ - \inf \left\{  \sum_{i=1}^{m+1} q(z_i) : \sum_{i=1}^{m+1}
                z_i > y,\, 0 < z_1,\ldots, z_{m+1} 
    \leqslant b  \right  \}  \right \}.
 \end{align*}
 
Since $q(\cdot)$ is increasing and concave, for $y\in \bigl(mb,
(m+1)b\bigr]$, the  infimum is achieved at, say, $z_1=\cdots=z_m=b$,
$z_{m+1} =y-mb$. Since for $b>1$, $q(b)\leq C_1b^{2\alpha}$ for some
$C_1>0$,  this establishes
  (\ref{eq;CvxAls0.0}) in the case $d=m+1$ with $\gamma_m$ 
and $c_m$ that must be at least $2\alpha$ and 
  $CC_1$, correspondingly. Their final values will be set in the
  sequel. 
  
  We continue to induct on $d$ while keeping the same 
  $m$.  Assume, therefore, that (\ref{eq;CvxAls0.0}) is valid for  
$d=m+1,\ldots,m+\ell$, some $\ell \geqslant 1$. In the case
$d=m+\ell+1$ write for $y\in \bigl( mb,(m+1)b\bigr]$, 
in the obvious notation 
\begin{align*}
    \PP\left\{ \sum_{i=1}^{d}  {\xi}_i \geqslant y  \right\}
    = & \int_{1} ^ {b} 
    F_{{\xi}} ( d z ) 
    \PP\left\{ \sum_{i=1}^{d-1} {\xi}_i \geqslant y - z  \right\} \ =:
        T_1 + T_2, 
\end{align*}
where $T_1$ and $T_2$ are the integrals over $(1, y-mb]$ and
$(y-mb,b)$, correspondingly. 
To estimate $T_1$, note that in this range $y-z>mb$, so we may use the
inductive assumption over $d$ to obtain 
$$
     \PP\left\{ \sum_{i=1}^{d-1}  {\xi}_i \geqslant y - z  \right\}  
     \leqslant  c_{m}^{d-1} b^{\gamma_{m}} \left( \overline{H}  (b)  \right) ^{m} \overline{H} \left (y - z - mb \right) .
$$
By (\ref{eq;CvxAls1.0})  
\begin{align*}
     \int_{1} ^ {y - m b} 
   F_{\xi} ( d z )  \overline{H} \left (y - z - mb \right)   
    \leqslant &  C \int_{1} ^ \infty
    H ( d z )  \overline{H} \left (y - z - mb \right) \\
    \leqslant & ( c_0^2 C) \overline{H} \left (y - mb\right).
\end{align*}
It follows that 
\begin{equation} \label{eq;CvxAls2.1}
  T_1 \leqslant   ( c_0^2 C) \, c_{m}^{d-1} 
      b ^ {\gamma_{m}} 
     \left( \overline{H} (b)  \right) ^{m} \overline{H} (y - mb). 
\end{equation}

To estimate $T_2$, note that in this range $(m-1)b<y-z\leqslant mb$,
and we use the inductive assumption over $m$ to write 
$$
T_2 \leqslant   c_{m-1}^{d-1} b  ^{\gamma_{m-1}}
    \left( \overline{H}(b) \right) ^ {m-1}
      \int_ {y - m b} ^ {b}
    F_ {{\xi}} ( d z ) 
     \overline{H} \big(y  - (m-1)b - z\big) .
$$
Using the same optimization under concavity argument as above shows
that 
\begin{align*}
  & \int_ {y - m b} ^ {b}
    F_ {{\xi}} ( d z ) 
     \overline{H} \big(y  - (m-1)b - z\big) \leqslant
    C\overline{H^{*2}}\big(y  - (m-1)b \bigr) \\
    \leqslant  & C  \bigl(q(b)\bigr)^2  \overline{H}(b)
                 \overline{H}(y-mb). 
\end{align*}
Therefore, 
\begin{equation} \label{eq;CvxAls3.1}
 T_2 \leqslant   C c_{m-1} ^{d-1} b ^
     {\gamma_{m-1}} \big( q(b) \big) ^ 2
    \left( \overline{H}(b) \right) ^ {m}
    \overline{H}(y - mb).
  \end{equation}
It follows from \eqref{eq;CvxAls2.1} and \eqref{eq;CvxAls3.1} that to
complete the inductive argument we only need to make the final
selection of $\gamma_m$ and $c_m$ to be so large as to satisfy 
\begin{align*}
     c_{m}^{d} \geq  (c_0^2 C) c_{m}^{d -1} + C c_{m-1}^{d-1}, \ \  \ 
  b ^ {\gamma_{m}} \geq b ^ {\gamma_{m-1}}
      \big( q(b) \big) ^ 2.    
\end{align*}
Since this can clearly be done, this completes the proof of
(\ref{eq;CvxAls0.0}). 

$(\rom2)$:
Note that 
\begin{equation} \label{eq;CvxAls5.0}
  \sum_{\Gamma_j \geqslant n^ r }
  V\left( w_n/\Gamma_j \right) 1_{\{  0 \in I_{j;n} \}} 
\eid \sum_{i=1}^{N_n}  \xi_i,
\end{equation}
where $N_n$ is a Poisson random variable with mean
$\overline{\nu}(x_0)- n^r /w_n$,  
and $\{ \xi_i\}_{i\geqslant 1}$ is a family of $i.i.d.$ random
variables independent of $N_n$, whose law is the measure $\nu$
restricted to the interval $(x_0,\overline{z}_{n})$ and normalized to
a probability measure there. Because of the range of $r$ and
\eqref{eq;asyV&G0.0} we see that  for large $n$ the event $\bigl\{
\sum_{i=1}^{d} {\xi_i} > z_n \}$ requires $d$ to be at least $m+1$. 
Therefore, in the notation of the first part of the lemma, 
by \eqref{eq;CvxAls0.0}, 
 \begin{align*}
     \PP \left\{  \sum_{i=1}^{N_n}  {\xi}_i > z_n \right \} 
    = &  \sum_{d=m+1}^ \infty 
    \PP \left\{  \sum_{i=1}^{d}  {\xi_i} > z_n \right \}  
    \PP \left\{ {N_n}  = d \right \} \\
    \leqslant &  \sum_{d=m+1}^ \infty 
    \frac{
     c_m ^ d  \overline{z}_n ^ {\gamma_m} \big(\overline{H}(\overline{z}_n) \big)^ m  \overline{H}(z_n - m \overline{z}_n)  
    }{d\,!} \\   
     \leqslant & e^{c_m}\overline{z}_n ^ {\gamma_m} 
     \big(\overline{H}(\overline{z}_n) \big)^ m  \overline{H}(z_n - m \overline{z}_n), 
 \end{align*}
 Since $\overline{z}_n\to\infty$, using $\gamma_m+1$ from the first
 part of the lemma as  $\gamma_m$ in  
 (\ref{eq;CvxAls0.1})   gives us (\ref{eq;CvxAls0.1}).
\end{proof}


\bibliographystyle{authyear}
\bibliography{CHEN}

\end{document}